\documentclass[12pt]{amsart}\usepackage{amssymb, graphicx, epsfig,verbatim}

\title[Lipeomorphic p-adic equivalence] {Lipeomorphic equivalence for p-adic analytic functions: a comparison between complex and p-adic dynamics}
\author{ Adrian Jenkins \and Steven Spallone }
\address{Department of Mathematics, Kansas State University, Manhattan,
KS, 66506} \email{majenkin@math.purdue.edu}
\address{Department of Mathematics, University of Oklahoma, Norman,
OK, 73019} \email{sspallone@math.ou.edu}

\numberwithin{equation}{section}

\newtheorem{thm}{Theorem}[section]

\newtheorem{cor}[thm]{Corollary}
\newtheorem{prop}[thm]{Proposition}
\newtheorem{define}[thm]{Definition}

\newtheorem{lemma}[thm]{Lemma}

\newcommand{\nc}{\newcommand}
\nc{\dmo}{\DeclareMathOperator}
\nc{\ul}{\underline}
\nc{\F}{\mathbb{F}}
\nc{\N}{\mathbb{N}}
\nc{\Zp}{\mathbb{Z}_p}
\nc{\Qp}{\mathbb{Q}_p}
\nc{\Cp}{\mathbb{C}_p}
\nc{\half}{\frac{1}{2}}
\nc{\W}{\ul{\N}}
\nc{\Q}{\mathbb{Q}}
\nc{\A}{\mathcal{A}}
\nc{\Z}{\mathbb{Z}}
\nc{\R}{\mathbf{R}}
\nc{\C}{\mathbf C}
\nc{\Berk}{\mathcal{B}}
\nc{\orb}{\mathbb{O}}
\nc{\isom}{\stackrel{\sim}{\to}}
\nc{\OK}{\mathcal{O}_0^K}
\nc{\eps}{\varepsilon}
\nc{\ra}{\rightarrow}
\nc{\gm}{\gamma}
\nc{\ol}{\overline}
\nc{\pp}{\mathfrak p}
\dmo{\ord}{ord}
\dmo{\Mod}{mod}
\dmo{\Span}{Span}

\begin{document}
\begin{abstract}
Let $K$ be a $p$-adic field, and suppose that $f$ and $g$ are germs of analytic functions on $K$ which are tangent to the identity at $0$.  It is known that $f$ and $g$ are homeomorphically equivalent, meaning there is an invertible germ $h$ so that $h^{-1} \circ f \circ h=g$.  In this paper, we determine whether there exists such $h$ which are lipeomorphisms, and moreover find the best possible H\"{o}lder estimate at $0$.  Our results have striking complex and real counterparts.
\end{abstract}

\maketitle
\section{Introduction}\setcounter{equation}{0}\label{section:s1}
Let $K$ be a $p$-adic field.  Consider an power series $f(x)=ax + \cdots$ with coefficients in $K$ and a nonzero disk $U$ of convergence at $0$, with $a \neq 0$.  These are $K$-analytic functions with a fixed point at $0$, considered as invertible germs at $0$.  A basic problem is to determine when two such germs are conjugate by a germ $h$ under the operation of composition.  There are several such problems, depending on what conditions we impose on $h$.  We may merely ask for $h$ to be homeomorphism, a rather weak condition, or insist on it being analytic, the strongest such condition.

 Homeomorphic equivalence depends on the dynamic nature of $a$, which is called the multiplier of $f$.  We consider four types of multipliers.  We say that $a$ (or $f$) is contracting if $|a| <1$, expanding if $|a|>1$, or indifferent if $|a|=1$.  In the last case, if $a$ is a root of unity, then $f$ is called indifferent rational; otherwise it is called indifferent irrational.

The set of contracting (resp. expanding) maps is a homeomorphic equivalence class.
The set of indifferent maps $f$ with no power of $f$ equal to the identity  is a homeomorphic equivalence class.  (We will encounter the proofs of most of these facts in this paper.)

As for analytic equivalence, the easiest case is when the multiplier of $f$ is not indifferent rational; then $f$ is analytically equivalent to $L_a(x)=ax$  (\cite{Herman-Yoccoz}).  The indifferent rational case reduces to the case where the multiplier is $1$.  Then, any such map is analytically equivalent to one of the form $f(x)=x+bx^m+cx^{2m+1}$.  (See \cite{Rivera-Letelier}, or \cite{Jenkins-Spallone} for another approach.)

The main object of this work is to study an intermediate equivalence relation between these two extremes, that of lipeomorphic equivalence.

\begin{define} Let $h$ be a homeomorphism.  Then $h$ is called a lipeomorphism if both $h$ and its inverse $h^{-1}$ are lipschitz continuous.  Two germs $f$ and $g$ are called lipeomorphically equivalent if they may be intertwined by a lipeomorphic germ.
\end{define}

In addition to the construction of lipeomorphic intertwining maps, a main goal of this paper is to show how imposing Lipschitz or H\"{o}lder continuity conditions on the intertwining maps $h$ leads to interesting new theory.  As we will see, these conditions illuminate the dynamical local nature of the functions.

Our primary focus is on maps whose multiplier is $1$.
Writing now $f(x)=x+ a_2x+ \cdots$, let $m$ be the smallest integer so that $a_{m+1} \neq 0$; we call $m$ the order of $f$.  This is an analytic invariant. Our first result is that if the order of $f$ is equal to the order of $g$, and the coefficients $a_{m+1}$ agree, then $f$ and $g$ are lipeomorphically equivalent.

\begin{thm}\label{lipeomorphic conjugacy}
Let $K$ be a $p$-adic field, and let $f(x)=x+ax^{m+1} + \cdots$ and $g(x)=x+ax^{m+1} + \cdots$ be analytic functions, with $a \neq 0$.  Then, $f$ is lipeomorphically equivalent to $g$.
\end{thm}

In fact, the conjugating map $h$ that we construct is an isometry.

In proving this theorem, we will employ the time-$t$ maps of certain vector fields $V_{m+1}=x^{m+1}\frac{\partial }{\partial x}$. These flows are simple to iterate, and thus serve as a useful representative for the lipeomorphic equivalence class. They are obtained by solving the differential equation
\begin{equation}\label{vector fields}
\frac{dx}{dt}=x^{m+1}.
\end{equation}
For example, the time-$t$ map of $V_{2}= x^{2}$ is $f_{2,t}(x)=x/(1-tx)$. We have the power-series expansion $f_{2,t}(x)=x+tx^{2}\ldots $. In fact, it is easily shown that for any $m$, we have the power series expansion $f_{m+1,t}=x+tx^{m+1}+\cdots $.


Our proof follows much of the ideas of Shcherbakov \cite{Shch} in the complex case.  Briefly, the idea is to transport the problem at $0$ to a problem at infinity by conjugating by an $m$th power map.  Let $a \in K$ and write $t_a$ for the translation function $ \eta \mapsto \eta+a$. After this transportation, we are left showing that a function $\tilde{f}$ which is suitably close to $t_a$ is lipeomorphically equivalent to $t_a$.  Write $G$ for the difference $G=\tilde{f}-t_a$, a rapidly decaying function.  The essential calculation of our proof is some light arithmetic, which yields $p$-adic estimates for sums of iterates of $G$.

In the second part of the paper, we constrain ourselves to $\Qp$, but expand our vision to comparing all the types of functions mentioned above to each other.  For instance, if the orders of $f$ and $g$ differ, we find that no intertwining map $h$ can be a lipeomorphism. One can be found, however, which satisfies a H\"{o}lder-type condition at $0$.  In fact, we find the best possible H\"{o}lder exponents at $0$ in all cases.  We prove the following theorem:

\begin{thm}\label{local Holder continuity}  Let $f(x)=x+ax^{m+1}+ \cdots$, $g(x)=x+a'x^{m'+1}+ \cdots$ be $\Qp$-analytic functions with $a,a' \neq 0$. Then $f$ and $g$ are conjugate via a homeomorphism $h$ defined between neighborhoods $U_{1}$ and $U_{2}$ of $0$. Suppose that $m \leq m'$.  Then we may choose $h$ and a constant $C>0$ with the property that
\begin{equation} \label{parade}
|h(x)| \leq C|x|^{\frac{m}{m'}}.
\end{equation}
This exponent cannot be improved.
Secondly, if $x,y \in U_1$ are in the same $\Zp$-orbit, then
\begin{equation} \label{furrst}
|h(y)-h(x)| \leq C |y-x|^{\frac{m}{m'}}.
\end{equation}
Finally, if $x,y \in U_2$ are in the same $\Zp$-orbit, then
\begin{equation} \label{sekund}
|h^{-1}(y)-h^{-1}(x)| \leq C |y-x|^{\frac{m}{m'}}.
\end{equation}
\end{thm}

We arrive at this theorem by analyzing the closures in $\Qp$ of the orbits of points close to $0$, which we call ``$\Zp$-orbits''.  The number of $\Zp$-orbits of a given norm for $f$ and $g$ determine how rapidly $h$ needs to shrink or expand, and the result follows.

One of the most interesting aspects of the two theorems cited above is the affinity with known results in the archimedean theory. We will discuss this more fully in the next section, but the guiding principle we have used here is that, after applying appropriate estimates, the non-archimedean theory closely mirrors its archimedean counterpart. This is quite surprising, particularly given the great differences between the topologies on the fields $\Qp$ and that on ${\mathbf C}$.

Next, we answer these same questions for the multiplier maps $L_a$ in $\Qp$.
As in the previous case, a line is drawn quite naturally when we ask for the best H\"{o}lder estimate at $0$ for the intertwining maps $h$ between different $L_a$.  Between any two contracting multiplier maps, and between any two indifferent multiplier maps, we determine this largest exponent $\alpha$.
For $a$ indifferent, write $N(a)$ for the number of $\Zp$-orbits of $L_a$ of each radius.
For indifferent multiplier maps, we prove that $L_a$ and $L_{a'}$ are lipeomorphically equivalent if and only if $N(a)=N(a')$.  Indeed the best exponent $\alpha$ is the ratio of these two numbers.  We compute $N(a)$ in Corollary \ref{multorbs}.

Moreover we show that a conjugating map from a flow $f_{m+1,a}$ to an indifferent irrational multiplier map cannot satisfy {\it any} H\"{o}lder estimate at $0$.  Thus, a map $f(x)=x+ax^m+ \cdots$ with $a \neq 0$ is not lipeomorphically linearizable.

 As an example of how homeomorphic equivalence is not without merit, we consider functions $f(x)=x + a_{m+1}x^m+ \cdots$ defined on different $p$-adic spaces, for distinct prime numbers $p$.  In fact, they are {\it never} homeomorphically equivalent, even though the spaces themselves are homeomorphic.

The paper is organized as follows: we begin with a short survey of known results in both the archimedean and non-archimedean categories in Section \ref{section:s1.5}.

Section \ref{section:s2} is devoted to some preliminaries on the additive and multiplicative theory for a general $p$-adic field. In section \ref{section:s3} we define the time-$a$ maps of the vector fields $V_{m+1}$ within the fields $\Qp$, as well as some topological analysis of the structure of their orbits.  Section \ref{section:s4} proves the lipeomorphism result of Theorem \ref{lipeomorphic conjugacy}.  In covering all $p$-adic fields, and all orders, it necessarily takes some effort.  The reader will find many of the proofs much easier in the case where $K=\Qp$ and when $m=2$.  On the other hand, the theory extends all the way to $K=\Cp$.  We also show that the conjugating lipeomorphisms, which are defined on disks, even extend to the corresponding Berkovich disk.

The rest of the paper has a cosier flavor, as we focus on $K=\Qp$.  It may be read independently of the earlier proofs. Section \ref{section:s5} is used to describe a general theory of ``bullseye spaces'', the kinds of topological spaces one gets from the set of orbits for our functions, as we take aim at behavior near $0$.  Their theory is used to give a detailed analysis of pointwise estimates of conjugating homeomorphisms between maps defined locally on $\Qp$.  Within the section can be found the proof of Theorem \ref{local Holder continuity}.

In Section \ref{Multipliers} we find the best H\"{o}lder constants for intertwining maps between multiplier maps, again using the theory of bullseye spaces.  Section \ref{section:s6} settles homeomorphic equivalence for functions defined on different $p$-adic spaces. Finally, we include an appendix, giving a technical proof of a result which is needed for Theorem \ref{lipeomorphic conjugacy}.

This project began while the authors were visiting Purdue University.  The research continued while the first author visited Kansas State University and the second visited the University of Oklahoma.  The authors are thankful for the support of these departments.  In addition, the authors would like to thank Rob Benedetto for a series of communications on the topics mentioned here, and in particular for suggesting the material in Section \ref{notstolovitch}.

\section{History}\setcounter{equation}{0}\label{section:s1.5}

It is worth summarizing some of the salient points of the archimedean theory, for if the field $K$ was chosen instead to be ${\mathbf R}$ or ${\mathbf C}$, this subject has been developed for many years by a variety of mathematicians.  One of the first attempts at understanding the topological structure of a holomorphic (i.e. analytic in ${\mathbf C}$) mapping $f(z)=z+\sum_{n\geq m+1}a_{n}z^{n}$ tangent to the identity and locally defined near $0\in {\mathbf C}$ was given by Fatou \cite{Fatou} in the late 1910s. While a topological conjugacy was not constructed on a full neighborhood of $0$, Fatou discovered that the orbits possessed a ``flower'' structure with a number of ``petals'' dependent only on the number $m$. Using this geometric argument, he was able to conclude that for $K={\mathbf C}$, the number $m$ is a topological invariant. 

\begin{figure}[h]
\epsfig{file=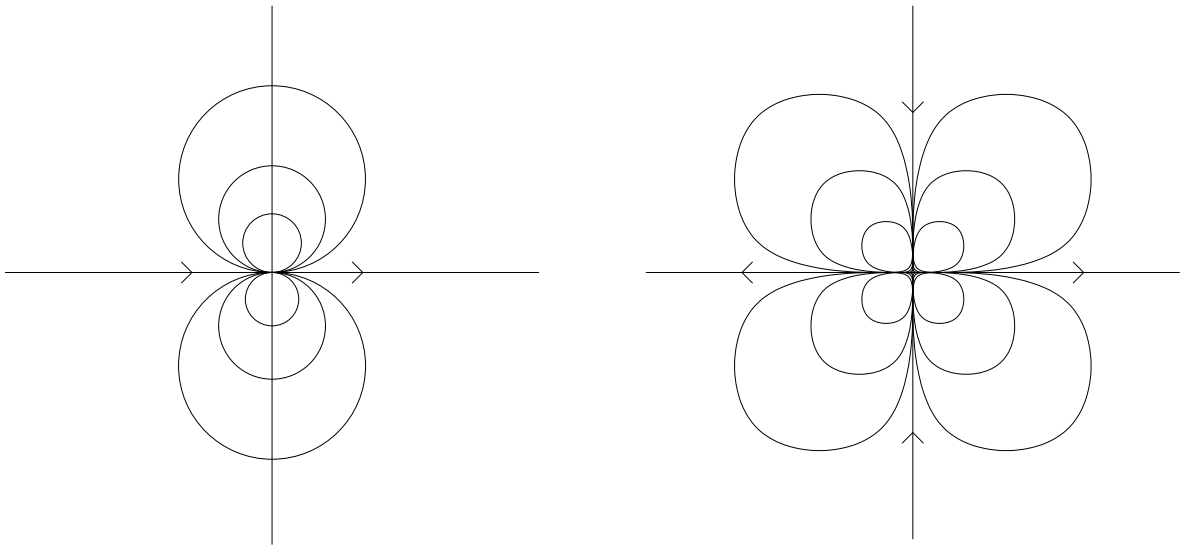,scale=.71}
\caption{The complex orbit structure of the maps $f_{2,1}$ and $f_{3,1}$. Note the differing numbers of ``petals" for each of the ``flowers".}
\end{figure}

However, it was not until the late 70's that it was shown that $m$ is the {\it only} topological invariant. The following result, proven independently by Camacho and Shcherbakov, demonstates this:
\newline

\noindent{\bf Theorem (Camacho \cite{Camacho}, Shcherbakov \cite{Shch})}. {\it Let $2\leq m\in {\N}$, and suppose that
$f(z)=z+az^{m+1}+\cdots $ is a function defined and holomorphic in a small neighborhood $U$ of the origin $0\in {\mathbf C}$, with $a\neq 0$. Then, there is a
neighborhood $0\in U'\subset U$ and a homeomorphism $h$ defined in $U'$ so that $(h^{-1}\circ f\circ h)(z)=\pi _{m+1}(z)=z+z^{m+1}$}.\newline
\noindent

We will speak more on Shcherbakov's version of this result in a moment. However, we note that a simple computation shows that the time-$1$ map $f_{m+1,1}$ of $V_{m+1}$ can be expanded as a power series centered at $0\in {\mathbf C}$ with only real (indeed, only rational) coefficients. Thus, we can also view these mappings as real-analytic diffeomorphisms defined on an interval $(-\delta ,\delta )$. However, the dynamics of these mappings as germs defined in $({\mathbf R},0)$ are, predictably, much less interesting than their complex counterparts. We recall a simple argument that demonstrates this. \newline

\noindent{\bf Lemma}. {\it Fix $m, n\in {\Z}$ with $m\geq 1, n\geq 1$. For a sufficiently small  $0<\varepsilon << 1$, the maps $f_{m+1,1}$ and $f_{n+1,1}$ are topologically equivalent on the interval $[0,\varepsilon ]$}.\newline

\begin{proof}
The change of variable $h_{m+1}(x)=mx^{m}$ conjugates the map $f_{2,1}$ to $f_{m+1,1}$. The result now follows.
\end{proof}

Other variations of this result are possible (e.g. by considering only odd $m$, one can construct homeomorphisms between different maps $f_{m+1,1}$ and $f_{n+1,1}$ on a open interval centered at $0$), but the point is this: the number $m$ is no longer a topological invariant when the maps $f$ are considered as real-valued, local real-analytic diffeomorphisms. Note, however, that the conjugating maps (as described in the Lemma above) possess the same H\"{o}lder-type estimates as those proven in Theorem \ref{local Holder continuity}.

Now, we further expound upon Shcherbakov's work. Let $f(z)=z+\sum_{n\geq m+1}a_{n}z^{n}$, with $a_{m+1} \neq 0$. The construction of the homeomorphism of Shcherbakov yields strong estimates which must be satisfied by $h$. In fact, fix $\varepsilon >0$ and $\delta >0$. Then, the neighborhood
$U'$ and the local conjugating homeomorphism $h$ can be chosen so that $h(z)=z+\ol h(z)$, where
\begin{enumerate}
\item
$|\ol h(z)|\leq |z|^{m+1-\varepsilon }$ for all $z\in U'$, and
\item
$|\ol h(z_{2})-\ol h(z_{1})|\leq \delta |z_{2}-z_{1}|$ for all $z_{1},z_{2}\in U'$.
\end{enumerate}

In other words, the number $m$ is a (complex-analytic) locally-lipeomorphic invariant. We will show that these estimates hold for our conjugating homeomorphisms in our proof of Theorem \ref{curbed}. In fact, the proof of Theorem \ref{lipeomorphic conjugacy} is strongly rooted in the techniques elaborated in the paper of Shcherbakov \cite{Shch}.

\section{Preliminaries and Notation}\setcounter{equation}{0}\label{section:s2}

\subsection{Additive Theory of $p$-adic Fields} \label{Additive elements}
In this paper we write $\N$ for the set of positive integers, and $\W$ for the set of nonnegative integers.

We will discuss two kinds of fields in this paper:

First we have the $p$-adic numbers $\Qp$ for a given prime integer $p \in \Z$.  If $a \in \Z$ is nonzero, let $\ord_p(a)$ be the maximal power of $p$ which divides $a$.  Then $\Qp$ may be defined as the completion of the rational numbers $\Qp$ under the norm $\left| \frac{m}{n} \right|=p^{-\ord_p(m)+\ord_p(n)}$.  This norm (and its extension to $\Qp$) is non-archimedean in the sense that one has $|x+y| \leq \max(|x|,|y|)$ for all $x,y \in \Qp$.  The subring $\Zp$ of $\Qp$ is given by those $p$-adic numbers whose norm is no greater than $1$.  One calls $\Zp$ the ring of $p$-adic integers.

Secondly we have the general case of a $p$-adic field $K$, which we define now.  The reader should be cautioned that there are variations of this definition in the literature; we have chosen one which is convenient for our purposes.

\begin{define} \label{pdef} Let $p$ be a prime number.  By a $p$-adic field $K$, we mean a complete non-archimedean field $(K,| \cdot|)$of characteristic zero so that $|p|<1$. \end{define}

Note that $K$ contains a copy of $\Q$ by virtue of having characteristic zero.  Moreover since it is a complete field with $|p|<1$ it contains
a canonical copy of $\Qp$ as well.  So $K$ is a $p$-adic field if it is a complete field extension of the non-archimedean field $\Qp$.  We will assume for simplicity that $|p|=\frac{1}{p}$.  The field $\Cp$, defined as the topological completion of the algebraic closure of $\Qp$, is a $p$-adic field.

For an introduction and basic facts about non-archimedean fields, the interested reader can consult a
standard text, such as \cite{Koblitz},\cite{Schikhof}, or our earlier paper \cite{Jenkins-Spallone}.

\begin{define} Given a $p$-adic field $K$, let $\Delta=\{x \in K \mid |x| \leq 1 \}, \pp=\{x \in K \mid |x|<1 \}$, and $\Delta^\times=\{x \in K \mid |x|=1 \}$.  Also let $\Delta'=\Delta - \{ 0 \}$. For an integer $m \geq 1$, put $\mu_m(K)=\{ \zeta \in K \mid \zeta^m=1 \}$.  For a number $r>0$, and $a \in K$, write $D(a,r)$ for the disk $\{ x \in K \mid |x-a| < r \}$, $\ol{D}(a,r)$ for the closed disk $\{ x \in K \mid |x-a| \leq r \}$, and $C(a,r)$ for the circle $\{ x \in K \mid |x-a|=r \}$.
\end{define}
The subset $\Delta$ is a subring of $K$, with $\pp$ as maximal ideal and $\Delta^\times$ as the group of units.  The open set $\Delta'$ is a punctured neighborhood of $0$.

Let $a \in K$ and consider the translation map $t_a: \eta \mapsto \eta +a$.
Our next task in this section is to pave the way for a study of the dynamics of $t_a$.  We need a good vector space complement to the canonical subspace $\Qp \cdot a$ inside $K$, and an estimate for how the norms fit together.

More generally, let $V$ be a $\Q_p$-vector space with a non-archimedean norm.  Write $\Delta_V= \{ v \in V \mid |v| \leq 1 \}$.  Then $\Delta_V$ is a
$\Zp$-module and $\Delta_V/p \Delta_V$ is a $\F_p$-vector space.

The following proposition can be found in \cite{Bourbaki1} (as Exercise 7 in Section 1 of Chapter I).

\begin{prop} \label{Bourbaki}
Let $(e_{\lambda})_{\lambda \in L}$ be a family of elements of $\Delta_V$ such that the images of $e_{\lambda}$ in $\Delta_V/p \Delta_V$ form a
basis of this $\F_p$-vector space.  Then the set $(e_{\lambda})_{\lambda \in L}$ is linearly independent and the vector space $F$ of $V$
generated by $(e_{\lambda})$ is dense in $V$.  If we put $|v|_1= \max_{\lambda} |\xi_\lambda|$ for every $v=\sum_{\lambda}  \xi_\lambda
e_\lambda \in F$, then $\frac{1}{p}|v|_{1}\leq |v|\leq |v|_{1}$.
\end{prop}

We will apply this result to a $p$-adic field $K$, which is certainly a $\Q_p$-vector space.
Let $a \in K^\times$.  We will set up a norm $| \cdot|_1$ on $K$ based on $a$.

\begin{define} Pick $a_0 \in \Delta - p \Delta$ so that $a=a_0p^i$ for some integer $i$. Let $e_{\lambda_0}=a_0$.
We may extend $e_{\lambda_0}$ to a set $(e_\lambda)$ as in Proposition \ref{Bourbaki}.  Thus we may write $F$ as the direct sum of $\Q_p \cdot a$ and the $\Q_p$-subspace $X_0$ of $F$ generated by the
other $e_\lambda$.  Write $X$ for the completion of $X_0$ inside $K$.  It is clear that $K$ is the direct sum of $\Q_p \cdot a$ and $X$.  Define $|v|_1= \max_{\lambda} |\xi_\lambda|$ for every $v=\sum_{\lambda}  \xi_\lambda e_\lambda \in F$;
this extends continuously to a unique norm $|v|_1$ on $K$. \end{define}

For example, $|a|_1=p^{-i}$.  Then we have
\begin{lemma} \label{independent}
If $\xi \in \Q_p$ and $x \in X$, then $|\xi a +x|_1=\max (|\xi|p^{-i}, |x|_1)$.
\end{lemma}

\begin{proof}
The lemma is immediate if we replace $X$ with $X_0$.   For the general case, if $x_n$ is a sequence in $X_0$ converging to $x \in X$, then
\begin{equation}
\begin{split}
 |\xi a+x|_1 &= \lim_n |\xi a+x_n|_1 \\
 &= \lim_n \max (|\xi|p^{-i}, |x_n|_1)\\
 &= \max ( |\xi|p^{-i}, |x|_1 ). 
\end{split}
\end{equation}
\end{proof}

\subsection{Multiplicative Theory of $p$-adic Fields}\label{decomposition}
Recall that as a topological group, we have the basic decomposition $\Qp^\times \cong  <p> \times \Zp^\times$, where $<p>$ denotes the multiplicative subgroup generated by $p$.  We will produce a similar decomposition for a general $p$-adic field $K$; in particular for $p$-adic fields without a discrete valuation such as $\Cp$.  Rather than $\Zp^
\times$, we use a certain set of coset representatives $\Delta_\pi$ for $K^\times/<\pi>$ for any $\pi \in \pp$.
Our eventual goal here is to produce good fundamental domains $X_\zeta$ for the $m$th power map near infinity.

\begin{define} Pick a nonzero $\pi \in \pp$.  Let $\Delta_\pi= \{ x \in K \mid |\pi| < |x| \leq 1 \}$. \end{define}
The reader is urged to think of the case where $K=\Qp$ and $\pi=p$, in which case $\Delta_\pi=\Delta^\times=\Zp^\times$.  In general it is a subset of $\Delta$ containing $\Delta^\times$.

\begin{prop} \label{decomp1}  As topological spaces we have
\begin{equation} \label{dec}
K^\times  \cong < \pi> \times \Delta_\pi \text{ and } \Delta' \cong \{1, \pi,\pi^2, \ldots \} \times \Delta_\pi.
\end{equation}
\end{prop}

\begin{proof}
We prove the proposition for $K^\times$; the case of $\Delta'$ is similar.
Consider the map $\mu: < \pi> \times \Delta_\pi \to K^\times$ given by multiplication. Let $x \in K^\times$, and pick $m \in \Z$ so that
\[   |\pi|^{m+1} < |x| \leq |\pi|^m. \]
Then it is clear that $y=\frac{x}{\pi^m} \in \Delta_\pi$, and so we have $\mu(\pi^m,y)=x$.  Thus $\mu$ is surjective.  Next, suppose that $\mu(\pi^a,x)=\mu(\pi^b,y)$.  Then $yx^{-1}$ is a power of $\pi$ but
\begin{equation*}
|\pi|  < |yx^{-1}| < |\pi^{-1}|.
\end{equation*}
It follows that $yx^{-1}=1$ and so $a=b$. Therefore $\mu$ is a bijection.  It is clearly a continuous and open map as well, and we conclude that $\mu$ is a homeomorphism.

\end{proof}

Next we recall the following results about roots of units and power series.

\begin{prop} \label{Hensel}
Let $m \geq 1$, and $K$ a $p$-adic field.
\begin{enumerate}
\item Let $b \in \Delta^\times$ and $I \subsetneq \Delta$ an ideal.
Let $\alpha_0 \in \Delta$, and suppose that $\alpha_0^m \equiv b \Mod m^2I$.  Then there is a unique $\alpha \in \Delta$ so that $\alpha^m=b$ and $\alpha \equiv \alpha_0 \Mod mI$.

\item Let $a \in K$.  Then there is a unique analytic function $g(x) \in K[|x|]$ so that $g(x)^m=1-amx^m$ and $g(x)=1-ax^m+O(x^{m+1})$.
\end{enumerate}

\end{prop}

\begin{proof}  Statement (i) follows from the well-known Hensel's Lemma.  For the generality of this paper, we refer the reader to \cite{Bourbaki2} III, Section 4.5, Corollary 1.  Statement (ii) follows from Hensel's Lemma for power series.
\end{proof}

We denote the function $g(x)$ in (ii) throughout this paper as $\sqrt[m]{1-amx^m}$.  Here are some corollaries of Proposition \ref{Hensel} (i).
\begin{cor} \label{mI}
Let $a \in \Delta$, $m \geq 1$, and $n \geq 2$.  Suppose that $x \in m^2 \Delta$.  Then there is a unique $m$th root of $1-amx^n$, written $\sqrt[m]{1-amx^n}$, so that
\begin{equation*}
 \sqrt[m]{1-amx^n} \equiv 1-ax^n \Mod m x^{n+1}.
 \end{equation*}
Moreover,
\begin{equation*}
\left| \sqrt[m]{1-amx^n}-1 \right|=|ax^n|
\end{equation*}
\end{cor}

\begin{proof} This follows by applying Proposition \ref{Hensel} to $I=x^{n+1}\Delta$, together with some calculation.
\end{proof}

\begin{cor} \label{Hensel2}

Let $K$ be a $p$-adic field, and $m \geq 1$.
\begin{enumerate}
\item Let $b \in \Delta^\times$.  Then the equation $x^m=b$ has a solution in $\Delta$ if and only if it has a solution in $\Delta/m^2 \pp $.

\item The quotient map $\mu_m(K) \ra \left( \Delta/m \pp \right)^\times$ is injective.

\end{enumerate}

\end{cor}

\begin{define} \label{uzeta} Write $\ol{\Delta_\pi}$ for the image of $\Delta_\pi$ under the quotient map $\Delta \to \Delta/m \pp$.  It is clear that $\mu_m(K)$ acts on $\ol{\Delta_\pi}$; take a system of representatives for this action, and write $U_1$ for its preimage in $\Delta_\pi$.  For $\zeta \in \mu_m(K)$, let $U_\zeta=\zeta \cdot U_1$. \end{define}

\begin{prop} \label{decomp2} Each $U_\zeta$ is open and closed in $K$.  Moreover we have a disjoint union
\[ \Delta_\pi= \bigsqcup_{\zeta \in \mu_m(K)} U_\zeta. \]
If $u \in U_\zeta$ and $u-v \in m \pp$, then $v \in U_\zeta$.
\end{prop}

\begin{proof}  The first statement follows because the quotient topology on the ring $\Delta/m \pp$ is discrete.
The second follows from Corollary \ref{Hensel2}(ii).  The third statement is obvious from the construction of $U_\zeta$.
\end{proof}

\begin{define} For $\zeta \in \mu_m(K)$, let $K_\zeta= <\pi> \times U_\zeta$. \end{define}

The sets $K_\zeta$ make good fundamental domains for the $m$th power map, by the following proposition:

\begin{prop} \label{pedagogical} Each $K_\zeta$ is open and closed in $K^\times$.  We have a disjoint union
\[ K^\times= \bigsqcup_{\zeta \in \mu_m(K)} K_\zeta. \]
Moreover the map $x \mapsto x^m$ is injective on each $K_\zeta$. \end{prop}

\begin{proof} The first two statements follow from Propositions \ref{decomp1} and \ref{decomp2}.  For the last statement, suppose that $(\pi^nu)^m=(\pi^{n'}u')^m$, with $u,u' \in U_\zeta$.  We may assume that $\zeta=1$.  By the uniqueness of the decomposition (\ref{dec}), we may assume that $n=n'$ and so $u^m=(u')^m$.  It follows that there is an $m$th root of unity $\zeta'$ so that $u'=\zeta' u$.  Reducing modulo $m \pp$, we see that $\ol{u'}$ and $\ol{u}$ are in the same orbit of the action of $\mu_m(K)$ and it follows that $\zeta'=1$ (using the definition of $U_1$).
\end{proof}

Pick an integer $N$ large enough so that $\pi^N \in m \pp$.
Let $X_\zeta= \{ \pi^N, \pi^{N+1}, \ldots \} \cdot U_\zeta$.

\begin{prop} \label{Xprops1}
We have the following properties.

\begin{enumerate}
\item Each $X_\zeta$ is open and closed in $K^\times$.
\item The map $x \mapsto x^m$ is injective on each $X_\zeta$.
\item The union of the $X_\zeta$ is equal to $\pi^N \Delta'$.
\end{enumerate}
\end{prop}

\begin{proof}
This follows as in the previous proposition.
\end{proof}

In particular the $X_\zeta$ comprise open fundamental domains for the action of $\mu_m$ on a punctured neighborhood of $0$.

\subsection{Review of Analytic Conjugacy}

An analytic function is one defined locally by convergent power series.  Here we give a more precise definition.
Given a power series $f(x)=\sum_n a_nx^n \in K[|x|]$, its radius of convergence about $0$ is given by
\begin{equation*}\label{radius of convergence}
\rho=\left(\limsup_{n \ra \infty} \sqrt[n]{ |a_n|} \right)^{-1}.
\end{equation*}

\begin{define} \label{defineanalytic} The power series $f(x) \in K[|x|]$ is analytic (at $0$) if $\rho>0$.
\end{define}

As mentioned in the introduction, the analytic classification of analytic maps tangent to the identity admits simple invariants.  If one starts with a power series $f$ of the form $x+ax^{m+1}+ \ldots$, with $a \neq 0$, it is a straightforward exercise to clear out the next $(m-1)$ terms by simply conjugating by a polynomial.  This reduces the problem to power series of the form $f(x)=x+ax^{m+1}+bx^{2m+1}+ O(x^{2m+2})$.  One may then conjugate by a formal power series $h$ to simply obtain the polynomial $x+ax^{m+1}+bx^{2m+1}$. The following is known (see, e.g. \cite{Rivera-Letelier}):

\begin{thm}\label{classification within $K$}
Let $K$ be a $p$-adic field, and let $f(x)=x+ax^{m+1}+bx^{2m+1}+ O(x^{2m+2})$ and $g(x)=x+a'x^{n+1}+b'x^{2n+1}+ O(x^{2n+2})$ be $K$-analytic maps, with $a$ and $a'$ nonzero. Then, there is a $K$-analytic map $h$ tangent to the identity so that $(h^{-1}\circ f\circ h)(x)=g$ if and only if $m=n$, and there exists $c \in K^\times$ so that $c^ma'=a$ and $c^{2m}b'=b$. \end{thm}

In particular, any such analytic map is analytically equivalent to a polynomial of the form $x+ax^{m+1}+bx^{2m+1}$.

\section{The Time-$a$ Maps $f_{m+1,a}$}\label{section:s3}

\subsection{Theory for a general $p$-adic field $K$}

We now give an analysis of the maps $f_{m+1,a}$.  These naturally arise as the solution of the initial value problem $dx/dt=(x(t))^{m+1}, x(0)=x_{0}=x$, although we do not use this fact.

\begin{define} Let $K$ be a $p$-adic field, $m \geq 1$, and $a \in \Delta'$.
 Recall the power series $\sqrt[m]{1-amx^m}$ from Proposition \ref{Hensel}. Define
\begin{equation} \label{flowing}
\begin{split}
f(x)& =f_{m+1,a}(x)\\
&=\frac{x}{\sqrt[m]{1-amx^m}}\\
&=x+ax^{m+1} + \cdots.
\end{split}
\end{equation}
\end{define}

In particular, for $m=1$ and $a=1$, the function $f_{2,1}=\frac{x}{1-x}$ is the standard M\"{o}bius transform.

We consider these maps as functions defined for $x \in \pp$.  Note that $|f(x)|=|x|$ for such $x$.

In this section, we will establish some basic dynamic properties of these maps.
Later in this paper we will show that a general power series $g(x)=x+ \cdots$ is lipeomorphically equivalent to
some $f_{m+1,a}(x)$, and so will share these properties.

Note that iteration of these maps is especially simple; for $n \in \N$ we have:
\begin{equation*}
f^{\circ n}(x)=\frac{x}{\sqrt[m]{1-namx^m}}.
\end{equation*}
We can extend this to $z \in \Zp$ by simply defining
\begin{equation*}
f^{\circ z}(x)=\frac{x}{\sqrt[m]{1-zamx^m}}.
\end{equation*}

One computes that if $z_1,z_2 \in \Zp$, then
\begin{equation*}
f^{\circ z_1}(f^{\circ z_2}(x))=f^{\circ (z_1+z_2)}(x).
\end{equation*}

The set of such $\Zp$-iterates forms what is called a $\Zp$-orbit.

\begin{define} Let $x \in K$.  The $\Zp$-orbit of $x$ under $f$ is the set
\begin{equation*}
\orb(x)=\{f^{\circ z}(x) \mid z \in \Zp \}.
\end{equation*}
\end{define}

\begin{prop} \label{phis} Fix $0 \neq x \in \pp$.  Let $\varphi: \Zp \to \orb(x)$ be given by $\varphi(z)=f^{\circ z}(x)$. Then $\varphi$ is a lipeomorphism.  The function $\Phi: \pp \times \Zp \to \pp$ given by $(x,z) \mapsto f^{\circ z}(x)$ is uniformly continuous.
\end{prop}

\begin{proof}
An inverse map to $\varphi$ is given by
\begin{equation} \label{express}
\varphi^{-1}(y) =\frac{y^{m}-x^{m}}{amx^{m}y^{m}}.
\end{equation}
Moreover, one computes that
\begin{equation*}
|\varphi(z)-\varphi(z')|=|ax^{m+1}||z-z'|,
\end{equation*}
which implies that $\varphi$ is a lipeomorphism.

The statement about $\Phi$ is straightforward.

\end{proof}

\begin{cor} $\orb(x)$ is the closure of the set of iterates $\{f^{\circ n}(x) \mid n \in \W \}$.
\end{cor}

For later use we record the following lemma:

\begin{lemma} \label{lateruse}
We have
\begin{equation*}
\left| f_{m+1,a}^{\circ z}(x)-x \right| = |azx^{m+1}|.
\end{equation*}
\end{lemma}
\begin{proof} This follows immediately from Corollary \ref{mI}.
\end{proof}

\noindent {\bf Remark:} Of course we may also define $f^{\circ \alpha}$ for $\alpha \in K$ by the same formulas while suitably shrinking the domain.  In this way we see that the germ of $f$ has roots of all orders.

\subsection{Theory for $\Qp$}

The previous subsection elaborated the properties of the time-$a$ maps which are valid in any $p$-adic field $K$. We now refine our study of these maps within the fields $\Qp$. These results are not necessary for any of the theory of Section \ref{section:s4} or the proof of Theorem \ref{lipeomorphic conjugacy}. However, we will require this more detailed knowledge of the orbits of the time-$a$ maps when proving Theorem \ref{local holder continuity}.

\begin{prop} \label{Qp-orbits} Fix $x \in \pp$.  Then $\orb(x)=\{ y \in \pp \mid y\equiv x \Mod ax^{m+1}\}$.
\end{prop}

\begin{proof}
Let $y= f_{m+1,a}^{\circ z}(x)$.  By Corollary \ref{mI}, we have $\sqrt[m]{1-zamx^m} \equiv 1-zax^m \Mod ax^{m+1}$, and in particular, $y \equiv x \Mod ax^{m+1}$.  Conversely, suppose that $y \equiv x \Mod a x^{m+1}$.  We claim that the expression
\begin{equation*}
z =\frac{y^m-x^m}{amx^{m}y^{m}}
\end{equation*}
is then in $\Zp$, so that $\varphi(z)=y$.  This claim is equivalent to the claim that $amx^my^m|(y^m-x^m)$.
Since $ax^{m+1}|(y-x)$, we are reduced to proving $mx^{m-1}|(y^{m-1}+ \cdots+ x^{m-1})$.  But this follows because each of the term in the sum is equal to $x^{m-1}(1+x_i)$ for some $x_i$ satisfying $ax^m|x_i$.
\end{proof}

The $\Z_p$-orbits are closed in $\Q_p$, and partition each circle $C(0,r)$ in $\pp$. We show here that if $r=p^{-i}$ for $i\geq 1$, then the map $f_{m+1,a}(x)$ will have exactly $p^{im+\ord(a)-1}(p-1)$ orbits in the circle $C(0, r)$. Figure 2 depicts this for $f_{2,1}(x)=\frac{x}{1-x}$ and $p=2$.

\begin{figure}[h]
\epsfig{file=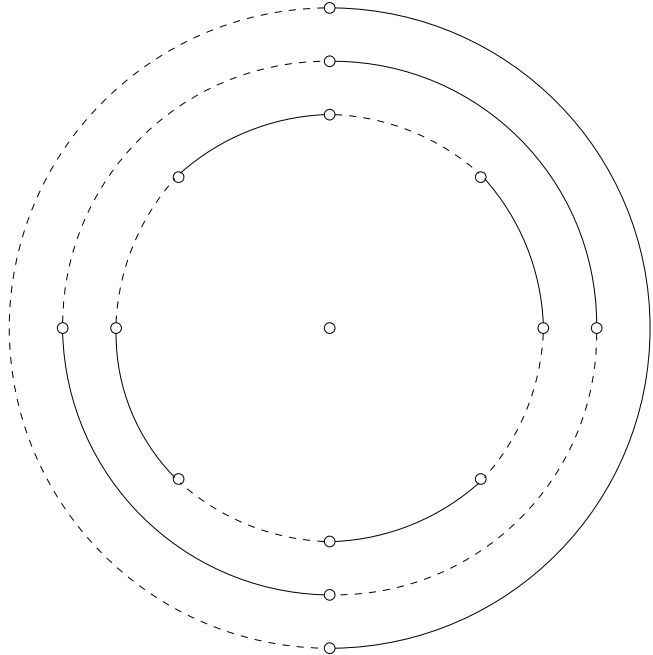,scale=.71}
\caption{$\Z_2$-orbits for the map $f_{2,1}$ within $\mathbb{Q}_{2}$. The image depicts three circular ``shells'', corresponding to radii $2^{-1}$, $2^{-2}$ and $2^{-3}$.}
\end{figure}

\begin{prop}\label{counting orbits}  Suppose that $\ord(a)=k$.
Then for $i \geq 1$, there are $p^{im+k-1}(p-1)$ orbits of the map $f_{m+1,a}$ in the circle $C(0, p^{-i})=p^i \Zp^\times$.
\end{prop}
\begin{proof}

Let $x\in \pp$.  If $|x|=p^{-i}$ for some (sufficiently large) $i\in {\N}$, then we can write $x=c p^{i}+\sum _{l\geq
i+1} c_l p^{l}$, where $c\in \{ 1, 2, \ldots, p-1\} $, and $c_{l}\in \{0, 1, \ldots, p-1\} $ for all $l$. The condition $y\equiv x\Mod ax^{m+1}$, this means that $y= c p^{i}+\sum _{i< l < i(m+1)+k} c_{l}p^{l} +O(ax^{m+1})$. Thus, by altering any of the coefficients $c, c_{i+1},\ldots ,c_{i(m+1)+k-1}$, we obtain a new orbit.
\end{proof}

Since there are but finitely many $\Z_p$-orbits in each circle, it now follows that each $\Z_p$-orbit is open, since each is the complement in an open circle of finitely many closed orbits.

\section{Lipschitz conjugacy between $f$ and $f_{m+1,a}$}\label{section:s4}

This section is devoted to the proof of Theorem \ref{lipeomorphic conjugacy}, an analogue of the result of Camacho and Shcherbakov. The focus here is not only on the construction of a homeomorphism conjugating $f$ and $f_{m+1,a}$, but also on its behavior.

\subsection{Strategy} \label{Gintro}

Before getting involved with the minutiae of the proof, we would like to sketch out the main ideas as a courtesy to the reader.  The goal is to find a lipeomorphism $h$ between a given function $f(x)=x+ax^{m+1}+ \cdots$ and $f_{m+1,a}$.  To do this directly is to solve the functional equation
\begin{equation} \label{recursive}
 h(f_{m+1,a}(x))=f(h(x)).
\end{equation}

 A first approach could be to pick orbit representatives of the $\Zp$-orbits $\orb$ for $f_{m+1,a}$, and attempt to define $h$ recursively through Equation (\ref{recursive}).  Thus for $z \in \Zp$ one might write
 \begin{equation*}
 h(f_{m+1,a}^{\circ z}(x))=f^{\circ z}(h(x)).
\end{equation*}
There are some problems with this approach.  By the previous section we know that the function $n \mapsto f_{m+1,a}^{\circ n}(x)$, defined for $n \in \N$, extends continuously to $z \in \Zp$.  But we do not know this for $f$ a priori.  (This will follow, for $x$ sufficiently close to $0$, after we have proven Theorem \ref{lipeomorphic conjugacy}.)  Another problem is that it is unclear with this strategy how one plans to define $h$ to be continuous as one passes through different $\Zp$-orbits, which are only open in the case of $K=\Qp$.

To bypass these difficulties we adapt a construction of Shcherbakov, who introduced this method in the context of $K=\C$.  Although we appreciate the intelligible behavior of the flows $f_{m+1,a}$, we can do even better by ``conjugating'' by the map $\A_{m+1}(x)=-\dfrac{1}{mx^m}$.  This transforms $f_{m+1,a}$ near $0$ to the simple translation $t_a(\eta)= \eta +a$ near infinity.  Note that the $\Zp$-orbits here are simply cosets $\eta_0+\Zp a$ of $\Zp a$.  Meanwhile it transforms the function $f$ to a map $\tilde{f}(\eta)=\eta +a + G(\eta)$ for a manageable function $G$.  We then apply the previously mentioned approach to the functional equation
\begin{equation*}
\tilde h(\eta+a)=\tilde h(\eta)+a+G(\tilde h(\eta))
\end{equation*}
to define a lipeomorphism $\tilde h$ at infinity.
We discover, through a pleasant application of the division algorithm, that this equation iterates well to give a Lipschitz map on $\Zp$-orbits.  Moreover there is a natural way to define $\tilde h$ to be continuous through different $\Zp$-orbits for general $p$-adic fields $K$; we made the delicate choices of basis $(e_\lambda)$ in Section \ref{Additive elements} to this end.

Coming back from infinity has its troubles, for the map $\A_{m+1}$ provides neither an injection nor a surjection from a neighborhood of $0$ to a neighborhood of infinity, if $m>1$.  This is not insurmountable: one divides up a punctured neighborhood of $0$ into open fundamental domains $X_\zeta$ for the action of the roots of unity $\mu_m$, which were introduced in Section \ref{decomposition}.  Then the restriction of $\A_{m+1}$ to each $X_\zeta$ is a diffeomorphism onto its image.  This image of $\A_{m+1}$ is certainly not a full neighborhood of infinity; we call such things ``$m$-neighborhoods of infinity'', and study them in the next section.  With this set-up we can pull back the lipeomorphisms at infinity to lipeomorphisms $h_\zeta$ on each sector $X_\zeta$.  Their only common limit point is $0$, and the maps easily paste together to yield the desired lipeomorphism $h$.

\subsection{The $m$-neighborhoods of infinity} \label{mnbhd}

In this section we give the details about the domain and target of $\A_{m+1}$ that we want.
Fix a nonzero element $\pi \in \pp$.

\begin{define} A subset $Y \subset K$ is an $m$-neighborhood of infinity if $Y=\A_{m+1}(\pi^N \Delta')$ for some $N \geq 0$ so that $\pi^{Nm} \in m \pp$. \end{define}

The assumption on $N$ is to ensure the following technical lemma.

\begin{lemma} \label{mlemma} Let $Y$ be an $m$-neighborhood of infinity, let $a \in \Delta$.  Then
$Y+ \Delta a=Y$.
\end{lemma}

\begin{proof}     A typical element of $Y$ may be written as $y=-\frac{1}{m}(\pi^i \alpha)^{-m}$, where $\alpha \in \Delta^\times$ and $i \geq L$.  Let $\beta \in \Delta$.
Then we have
\begin{equation*}
y+\beta a=-\frac{1}{m}\pi^{-Lm}(\alpha^{-m}+(-m)\pi^{Lm}\beta a).
\end{equation*}
This element is in $Y$ by Proposition \ref{Hensel} (i).
\end{proof}

Recall from Section \ref{decomposition} that we have defined a fundamental domain $\Delta_\pi$ for the action of $\pi$ on $K^\times$, and a fundamental domain $U_\zeta$ for the action of $\mu_m$ on $\Delta_\pi$.
For $\zeta \in \mu_m(K)$ and a large $N$, recall that we have open sets $X_\zeta=U_\zeta \times \{\pi^N,\pi^{N+1}, \ldots \}$.  Their union is $\pi^N \Delta'$.

Then the restriction of $\A_{m+1}$ to $X_\zeta$ gives a diffeomorphism from $X_\zeta$ onto its image $Y$.  (Injectivity is proven in Proposition \ref{Xprops1}.)

\begin{define} We may therefore define
\begin{equation*}
(\A_{m+1}^{-1})_\zeta: Y \isom X_\zeta
\end{equation*}
to be its inverse.
\end{define}

The reader should enjoy checking that $\A_{m+1} \circ f_{m+1,a} \circ (\A_{m+1}^{-1})_\zeta=t_a$ for all $\zeta$.

\begin{prop} \label{Left4Adrian} Let $f(x)=x+ax^{m+1}+ \cdots$, let $Y$ be an $m$-neighborhood of infinity, and let $\delta > 0$.  Put $\tilde f=\A_{m+1} \circ f \circ (\A_{m+1}^{-1})_\zeta: Y \to Y$.  Then there is a constant $C>0$ and an $m$-neighborhood of infinity $Y' \subseteq Y$ so that $\tilde f(\eta)= \eta+a+ G(\eta)$, where $G: Y' \to K$ is a function satisfying the estimates
\begin{equation} \label{G-estimate}
|G(\eta)| \leq \frac{C}{|\eta|}.
\end{equation}
For all $\eta_1, \eta_2 \in Y'$, we have
\begin{equation} \label{G-Lip-estimate}
|G(\eta_1)-G(\eta_2)| \leq  \delta |\eta_1 -\eta_2|.
\end{equation}
\end{prop}

\begin{proof}
Let $f(x)=x+ax^{m+1}+F(x)$, where $F$ is an analytic function defined in a neighborhood of $0$ satisfying $F(0)=0$ and $|F(x)|=O(x^{2m+1})$. We show here that the function $\tilde f=\A_{m+1}\circ f\circ (\A_{m+1}^{-1})_\zeta$ can be written
\begin{equation*}
\tilde f(\eta )=\eta +a+G(\eta )
\end{equation*}
where $|G(\eta )|=O(|\eta |^{-1})$ and $|G'(\eta )|=O(|\eta |^{-2})$.  This will imply the proposition.

To see that we may write $\tilde{f}$ in this way, consider the formula
\begin{equation*}
\tilde f(\eta ) =\frac{-1}{m\left(  (\A_{m+1}^{-1})_\zeta(\eta) +a  (\A_{m+1}^{-1})_\zeta(\eta) ^{m+1}+ F((\A_{m+1}^{-1})_\zeta(\eta)) \right)^{m}} .
\end{equation*}
By expanding the polynomial in the denominator, and using $(\A_{m+1}^{-1})_\zeta(\eta)^m=-\dfrac{1}{m\eta}$, we obtain
\begin{equation*}
\tilde f(\eta)= \frac{-1}{\left( -\frac{1}{\eta }\right) +a\left( -\frac{1}{\eta }\right) ^{2}+K(\eta )},
\end{equation*}
where $K$ is a differentiable function defined in an $m$-neighborhood of $\infty $ satisfying $|K(\eta )|=O(|\eta |^{-3})$ and $|K'(\eta )|=O(|\eta |^{-4})$.  This gives
\begin{equation*}
\tilde{f}(\eta)=\frac{\eta }{1-\left( \frac{a}{\eta}+\eta K(\eta)\right)}=\eta +a+\eta^{2}K(\eta )+\sum_{n=2}^{\infty } \eta \left( \frac{a}{\eta }+\eta K(\eta)\right)^{n}.
\end{equation*}

Now, via the estimates on $K$, it is clear that $G$ is a function with the desired estimates.
\end{proof}

\subsection{Defining $\tilde h$ on $\Zp$-orbits}

In the following proposition we begin to construct our conjugating map $\tilde h$ at infinity.  At this point we are defining it on individual $\Zp$-orbits.  This is the cornerstone of our proof of Theorem \ref{lipeomorphic conjugacy}, so we give some motivation.  Let $\tilde f_0(\eta)=\eta+a$.  For $\tilde h$ to conjugate $\tilde f$ to $\tilde f_0$ we must have
\[ \tilde h(\eta)+a+G(\tilde h(\eta))=\tilde h(\eta+a). \]
Writing $\tilde h(\eta)=\eta+\hat h(\eta)$, this translates to the condition
\[ \hat h(\eta+a)=\hat h(\eta)+G(\eta+\hat h(\eta)). \]
So once we have chosen an $\hat h(\eta)$, all values of $\hat h$ on $\eta+ \W a$ are obtained by iterating this equation.  Namely,
\begin{equation*}
\hat h(\eta+na)= \sum_{j=0}^{n-1} G(\eta+ja+\hat h(\eta+ja)).
\end{equation*}
For the following arithmetic proof, recall the divisibility notation $x|y$ when $x,y \in \Delta$.  This simply means that $\dfrac{y}{x} \in \Delta$, which is the same as $|y| \leq |x|$.

\begin{prop} \label{Ouija}     Let $K$ be a $p$-adic field, $a \in \Delta'$, $Y$ an $m$-neighborhood of infinity, and $k \geq 2$.  Let $G: Y \to K$ be a function satisfying (\ref{G-estimate}) and (\ref{G-Lip-estimate}) for some $C$ and $\delta=p^{-k}$.   By shrinking $Y$ if necessary we assume that
\begin{equation} \label{busdate}
 |G(\eta)| < \frac{1}{p^k}
\end{equation}
for all $\eta \in Y$.
Fix $\eta_0  \in Y$, and define a sequence $\hat{h}_n=\hat{h}_n[\eta_0;a]$ for $n \geq 0$ via $\hat{h}_0=0$ and
\begin{equation} \label{Adrsum}
\hat{h}_n= \sum_{j=0}^{n-1} G(\eta_0+ja+\hat{h}_j).
\end{equation}
Then
\begin{enumerate}
\item For all $n$, $\hat{h}_n \in (p^k)$, where $(p^k)$ is the ideal $p^k \Delta \subseteq K$.
\item If $m,n \in \N$ and $m \equiv n \Mod p^{\ell}$ then
$p^{\ell+k}| (\hat{h}_m-\hat{h}_n)$.
\end{enumerate}
\end{prop}

\begin{proof}
Statement (i) follows by (\ref{busdate}), noting that the arguments of $G$ in (\ref{Adrsum}) make sense inductively by (i) and Lemma \ref{mlemma}.  We prove the second statement by induction on $\ell$.  For $\ell=0$ we must show that $p^k| (\hat{h}_m-\hat{h}_n)$ for all $m,n$.  This follows from (i).
Now assume the statement for a given $\ell$; we now prove it for $\ell+1$.  Suppose that $p^{\ell+1}|(n-m)$, with $n \geq m$.
In particular $m$ and $n$ have the same remainder $r_0$ upon division by $p^\ell$.
Write $n=dp^\ell+r_0$  and $m=cp^\ell+r_0$ for $0 \leq r_0 < p^\ell$.  Note that $p|(d-c)$.

We need to show that
\begin{equation*}
p^{\ell+k+1} | \sum_{j=m}^{n-1} G(\eta_0+ja+\hat{h}_j).
\end{equation*}

Writing each $j=p^{\ell}q+r$ via the division algorithm, this becomes
\begin{equation} \label{ttiger}
p^{\ell+k+1} | \sum_{r=r_0}^{r_0+p^\ell-1} \sum_{q=c}^{d-1}  G(\eta_0+(p^\ell q+r)a +\hat{h}_{p^\ell q+r}).
\end{equation}
By induction we know that $\hat{h}_{p^\ell q+r} \equiv \hat{h}_r \Mod p^{\ell+k}$.  Therefore by Estimate (\ref{G-Lip-estimate}), which says that $|G(\eta_1)-G(\eta_2)| \leq \dfrac{1}{p^k}|\eta_1-\eta_2|$, the sum in (\ref{ttiger}) is congruent $\Mod p^{\ell+k+1}$ to
\begin{equation*}
\sum_{r=r_0}^{r_0+p^{\ell}-1} \sum_{q=c}^{d-1}  G(\eta_0 +ra +\hat{h}_r)=(d-c) \sum_{r=r_0}^{r_0+p^{\ell}-1} G(\eta_0 +ra +\hat{h}_r).
\end{equation*}
Since $p|(d-c)$, we reduce to proving that
\begin{equation*}
p^{\ell+k} |  \sum_{r=r_0}^{r_0+p^{\ell}-1} G(\eta_0+ra+\hat{h}_r)=\hat h_{p^\ell+r_0}-\hat h_{r_0},
\end{equation*}
which follows by induction.
\end{proof}

\begin{define} Given $\eta \in Y$, write $\orb(\eta)= \{ \eta + z a \mid z \in \Zp \}$.    \end{define}

Note that $\eta + \W a$  is dense in $\orb(\eta)$.  Recall that $\tilde f$ and $t_a$ satisfy $\tilde f(\eta)=\eta+a+G(\eta)$ and $t_a(\eta)=\eta+a$.

\begin{cor} \label{lip function}Define a function $\hat{h}: \eta_0+ \W a \to (p^k)$ via $\hat{h}(\eta_0+na)=\hat{h}_n[\eta_0,a]$.  Then $\hat{h}$ is a bounded Lipschitz function satisfying $|\hat{h}(\eta )|\leq p^{-k}$, with Lipschitz constant $\dfrac{1}{p^k}$.   Moreover it satisfies $\tilde f \circ h=h \circ  t_a$ on $\eta_0+\W a$.\end{cor}

\begin{proof}
The bound on $\hat h$ follows immediately from part (i) of Proposition \ref{Ouija}. Let $\eta_1, \eta_2 \in \eta_0+ \W a$.  Assume $\eta_1=\eta_0+ma$ and $\eta_2=\eta_0+na$ with
$m,n \in \W$ unequal.  Write $n-m=up^\ell$ with $p$ not dividing $u$ and $\ell \geq 0$.  Then
\begin{equation*}
\begin{split}
 |\hat{h}(\eta_2)-\hat{h}(\eta_1)| &= |\hat{h}_m-\hat{h}_n| \\
& \leq \frac{1}{p^{k+\ell}} \\
&= \frac{1}{p^k} |m-n| \\
&=\frac{1}{p^k}|\eta_2-\eta_1|, \\
\end{split}
\end{equation*}
where we have used Proposition \ref{Ouija} (ii) for the first inequality.  The functional equation follows from the evident equality
\begin{equation*}
\hat h_{n+1}[\eta_0,a]=\hat h_n[\eta_0,a]+G(\eta_0+na+\hat h_n[\eta_0,a]).
\end{equation*}
\end{proof}

\subsection{A Lipschitz map at infinity}
In this section we show how to define $\tilde h$ uniformly on all the $\Zp$-orbits in $Y$.
Write $K=\Q_pa \oplus X$ as in Definition \ref{independent}. Fix $\ell \in \N$ so that $p^{-\ell} \leq |a|$.
Let $Y$ be a sufficiently small $m$-neighborhood of infinity.  That is, $Y$ is small enough to satisfy the hypotheses of Proposition \ref{Ouija} for some $k >\ell$.

Choose a fixed set $D=\{ \eta_0 \}$ of representatives for cosets of $\Zp a$ in $\Qp a$.  (For instance if $a=1$ we may use the set of $p$-adic numbers with no integral part.) Then the set $D+\W a+X$ is dense in $K$.  Since $Y$ is open, it follows that $(D+\W a+X) \cap Y$ is dense in $Y$.

If $\eta_0 \in D$, $n \in \W$, $x \in X$, and $\eta_0+na+x \in Y$, define $\hat{h}(\eta_0+na+x)=\hat{h}_n[\eta_0+x;a]$ as in (\ref{Adrsum}).

\begin{prop} \label{Lipinfact} Pick $\ell \geq 1$ so that $p^{-\ell} \leq |a|$.   The function $\hat{h}$ as defined above is Lipschitz and in fact
\[ |\hat{h}(\beta)-\hat{h}(\alpha)| \leq \frac{1}{p^{k-\ell}}|\beta-\alpha|, \]
for all $\alpha,\beta \in (D+\W a+X) \cap Y$.
\end{prop}

\begin{proof}
Say $\alpha=\eta_0+ma+x$ and $\beta=\eta_0'+na+y$.  If $\eta_0' \neq \eta_0$ then
\begin{equation*}
\begin{split}
|\alpha-\beta| & \geq \frac{1}{p}|\alpha-\beta|_1 \\
&\geq  \frac{1}{p} |\eta_0-\eta_0'| \\
&> |a|. \\
\end{split}
\end{equation*}
The first inequality follows from Proposition \ref{Bourbaki}.  The second follows from Lemma \ref{independent}, since
$\left|\frac{\eta_0-\eta_0'}{a}\right| >1$.
  Since $\hat{h}(\alpha),\hat{h}(\beta) \in (p^k)$ by Proposition \ref{Ouija}, we have
\begin{equation*}
\begin{split}
 |\hat{h}(\beta)-\hat{h}(\alpha)|  &\leq \frac{1}{p^k} \\
                                   & < \frac{1}{p^k|a|}|\beta- \alpha|\\
                                   & \leq \frac{1}{p^{k-\ell}}|\beta-\alpha|. \\
 \end{split}
 \end{equation*}

So we may assume that $\eta_0=\eta_0'$.  We will show that
\begin{equation} \label{quick}
 |\hat{h}(\eta_0+ma+x)-\hat{h}(\eta_0+ma+y)| \leq \frac{1}{p^k}|x-y|
\end{equation}
and
\begin{equation} \label{draw}
|\hat{h}(\eta_0+ma+y)-\hat{h}(\eta_0+na+y)| \leq \frac{1}{p^k}|m-n|.
\end{equation}

Combining these we obtain
\begin{equation} \label{combined}
|\hat{h}(\alpha)-\hat{h}(\beta)| \leq \frac{1}{p^k} \max (|x-y|,|m-n|).
\end{equation}

By Proposition \ref{Bourbaki} and Lemma \ref{independent}, we have
\begin{equation*}
\begin{split}
 \max(|m-n|,|x-y|) &\leq \max(|m-n|,|x-y|_1) \\
& \leq |\alpha-\beta|_1 \\
& \leq p|\alpha-\beta|. \\
\end{split}
\end{equation*}

Together with (\ref{combined}) this gives
\[ |\hat{h}(\alpha)-\hat{h}(\beta)| \leq \frac{1}{p^{k-1}}|\alpha-\beta|, \]
giving the proposition.

We now prove inequality (\ref{quick}).  For $m=0$, we have $\hat{h}(\eta_0+x)=\hat{h}(\eta_0+y)=0$ by construction.  Now inductively assume the inequality is true up to $m$.  For $m+1$ we have
\[\left|\hat{h}(\eta_0+(m+1)a+x)-\hat{h}(\eta_0+(m+1)a+y)\right| =\left| \sum_{j=0}^m (G(\eta_0+ja+x+\hat{h}_j^x)-G(\eta_0+ja+y+\hat{h}_j^y))\right|, \]
where $\hat{h}_j^x=\hat{h}(\eta_0+ja+x)$ and $\hat{h}_j^y=\hat{h}(\eta_0+ja+y)$.
This expression is no greater than
\begin{equation*}
\begin{split}
 \max_j \left|G(\eta_0+ja+x+\hat{h}_j^x)-G(\eta_0+ja+y+\hat{h}_j^y)\right| &\leq \max_j \frac{1}{p^{k+1}}|x-y+\hat{h}_j^x-\hat{h}_j^y| \\
& \leq \frac{1}{p^{k+1}} |x-y|.  \\
\end{split}
\end{equation*}
We have used Proposition \ref{Left4Adrian} for the first inequality, using $\delta=\dfrac{1}{p^k}$, and our inductive hypothesis for the second inequality.  Equation (\ref{draw}) follows from Proposition \ref{Ouija}(ii) applied to $\eta_0+y$.

\end{proof}

\begin{thm} \label{mcgraw} Let $K$ be a $p$-adic field, let $Y$ be an $m$-neighborhood of infinity in $K$, and let $G$ be a $K$-valued function defined for $\eta \in Y$.  Suppose that there are positive constants $C,\delta>0$ so that for all $\eta \in Y$, $G$ satisfies estimates (\ref{G-estimate}) and (\ref{G-Lip-estimate}).

Define $\tilde f(\eta)=\eta+a+G(\eta)$.  Let $\eps >0$.  Then there exists an $m$-neighborhood of infinity $Y' \subseteq Y$ and a homeomorphism $\tilde h: Y' \to Y'$ of the form
\begin{equation*}
\tilde h(\eta)=\eta + \hat{h}(\eta),
\end{equation*}
so that
\begin{enumerate}
\item
$(\tilde h^{-1} \circ \tilde f \circ \tilde h)(\eta)=\eta+a. $

\item
$ |\hat{h}(\eta)| < \eps . $
\item
$ | \hat{h}(\eta_2)-\hat{h}(\eta_1)| < \delta |\eta_2-\eta_1|. $

\end{enumerate}
\end{thm}

\begin{proof}
Pick $k$ so that $p^{-k} \leq \eps$.  Shrink $Y$ so that (\ref{busdate}) holds.  Define $\hat h(\eta_0+na+x)$ to be $\hat{h}_n[\eta_0+x;a]$ for $\eta_0+na+x \in Y \cap (D+\N a+X)$.  Let $\tilde h(\eta)=\eta+\hat{h}(\eta)$ on this set.  By Proposition \ref{Lipinfact}, $\tilde h$ is Lipschitz on this domain, which is dense in $Y$.  It therefore extends continuously to $Y$, and satisfies the same estimates and functional equation there.
\end{proof}

\subsection{Proof of Theorem \ref{lipeomorphic conjugacy}}

Theorem \ref{lipeomorphic conjugacy} results from the following:

\begin{thm}\label{curbed}
Let $K$ be a $p$-adic field, and let $f(x)=x+ax^{m+1} + \cdots$ be an analytic function, with $a \neq 0$.  Then, $f$ is lipeomorphically equivalent to $f_{m+1,a}$.  More specifically,  for any $\eps >0$ and $0 < \delta < \half$, there is a neighborhood $X'$ of $0$ which is both $f$ and $f_{m+1,a}$-invariant, and a bijection $h: X' \to X'$ so that $h^{-1} \circ f \circ h=f_{m+1,a}$, and the following holds.  If we write
$h(x)=x+\ol h(x)$ then the map $\hat h$ satisfies the following estimates:
\begin{enumerate}
\item
$|\ol h (x) |\leq |x|^{m+1-\varepsilon }$;
\item
$|\ol h(x_{2})-\ol h(x_{1})|\leq \delta |x_{2}-x_{1}|$.
\end{enumerate}
Moreover if we write $h^{-1}=x+\ol k(x)$, then $\ol k$ has Lipschitz constant no greater than $\delta$.
\end{thm}
In particular $h$ is an isometry, and a lipeomorphism.

\begin{proof}
Fix $\pi \in \pp$.  By Theorem \ref{classification within $K$} we may assume that $f(x)=x+ax^{m+1}+b x^{2m}$ for $b \in K$. By conjugating by a dilation we may assume that $a \in \Delta$.  Pick an integer $N$ large enough so that $\pi^N$, $a \pi^{Nm}$, $b \pi^{2mN} \in m \pp$, and $|\pi|^N < \delta |m|$.
Let $X_\zeta= \{ \pi^N, \pi^{N+1}, \ldots \} \cdot U_\zeta$ as in Section \ref{mnbhd}.

\begin{lemma} \label{Xprops2}
 Each $X_\zeta$ is invariant under $f$.
\end{lemma}

\begin{proof}
Say that $x= \pi^Mu$, with $u \in U_\zeta$ and $M \geq N$.  We have

\begin{equation*}
\begin{split}
 f(x)   &=  \pi^M u+a(\pi^Mu)^{m+1}+ b (\pi^Mu)^{2m+1} \\
        &=  \pi^M (u+ a\pi^{Mm} u^{m+1} + b \pi^{2mM} u^{2m+1}). \\
 \end{split}
\end{equation*}
By Proposition \ref{decomp2}, the expression in parenthesis is in $U_\zeta$, and it follows that $f(x) \in X_\zeta$.  This concludes the proof of the lemma.
\end{proof}

Write $Y$ for the image of the $X_\zeta$ under $\A_{m+1}$.  Write $\tilde f_\zeta: Y \to Y$ for the function $\A_{m+1} \circ f \circ (\A_{m+1}^{-1})_\zeta$.

Let $\delta'>0$. By Proposition \ref{Left4Adrian} one shrinks the neighborhoods to obtain a Lipschitz constant of $\delta'$ for $G$.  Then, one has the hypotheses for Theorem \ref{mcgraw}.  We therefore obtain a lipeomorphism $\tilde h_\zeta(\eta)=x+\hat h_\zeta(\eta)$ on some $m$-neighborhood $Y' \subseteq Y$, invariant under both $\tilde f_\zeta$ and $t_a$, so that $\tilde h_\zeta^{-1} \circ \tilde f_\zeta \circ \tilde h_\zeta=t_a$.  Moreover $\hat h_\zeta$ has Lipschitz constant $\delta'$.  Write $X'$ for the preimage of $Y'$ under $\A_{m+1}$ and let $X_\zeta'=X' \cap X_\zeta$. Define $h_\zeta: X_\zeta' \to X_\zeta'$ via
\begin{equation*}
h_\zeta=(\A_{m+1}^{-1})_\zeta \circ \tilde h_\zeta \circ \A_{m+1}.
\end{equation*}
Then $h_\zeta$ conjugates $f$ to $f_{m+1,a}$ on $X_\zeta'$.  Write $h_\zeta(x)=x+\ol h_\zeta(x)$.
Now let $\delta>0$.  Possibly shrinking $X'$ and $Y'$ more, we have the hypotheses for Proposition \ref{Americanproof} in the Appendix, and may assume $\ol h_\zeta$ has a Lipschitz constant of $\delta$.

Then we simply define $h: X' \to X'$ via
\begin{equation*}
h(x)= \begin{cases} & h_\zeta(x) \text{ if } x \in X_\zeta', \\
              & 0, \text{ if } x=0. \end{cases}
\end{equation*}
Let $\ol h(x)=h(x)-x$. We claim that, for all $x_1,x_2 \in X'$, we have
\begin{equation} \label{forgotten}
|\ol h(x_2)- \ol h(x_1)| < \delta |x_2-x_1|.
\end{equation}
If $x_1,x_2$ lie in the same sector, or if either of them is $0$, then this follows from properties of $\ol h_\zeta$.  Otherwise let $x_1=\pi^{N_1}u_1$ and $x_2=\pi^{N_2}u_2$, with $u_1 \in U_{\zeta_1}$, $u_2 \in U_{\zeta_2}$, and $\zeta_1 \neq \zeta_2$.  If $N \leq N_1 < N_2$, then
\begin{equation}
|x_2-x_1|= |\pi|^{N_1}.
\end{equation}
If $N \leq N_1=N_2$, then
\begin{equation}
|x_2-x_1| = |\pi|^{N_1} |u_2-u_1| \geq |\pi|^{N_1}|m|,
\end{equation}
by Proposition \ref{decomp2}.

In both of these cases,
\begin{equation}
 |\ol h(x_2)- \ol h(x_1)| \leq |\pi|^{2N_1}.
\end{equation}
Since we have picked $N$ large enough so that $|\pi|^N < \delta |m|$, our claim (\ref{forgotten}) is proven in all cases.
Therefore $h$ is a lipeomorphism and $h^{-1} \circ f \circ h= f_{m+1,a}$ with $\ol h$ having Lipschitz constant $\delta$.
 Note that since $\delta<1$, this implies that $h^{-1}(x)=x+\ol{k}(x)$, where $\ol{k}$ has a Lipschitz constant $\frac{\delta}{1-\delta}$.
This concludes the proof of the theorem.
\end{proof}

\subsection{Extension to the Berkovich Disk} \label{notstolovitch}
In this section we add the condition that $K$ is algebraically closed, for example $K=\C_p$.  In modern-day $p$-adic dynamics, a closed disk $\ol D(a,r)$ in $K$ is considered as a dense subset of a path-connected compact space known as the closed Berkovich disk $\ol D^{\Berk}(a,r)$ (see \cite{Silverman} or \cite{Benedetto} for more details).  Briefly, the Berkovich disk is comprised of four different types of points.  Type I points are the original points of $\ol D(a,r)$.  Type II and III points, which we need not distinguish, correspond to closed disks $\ol D(b,s) \subseteq \ol D(a,r)$.  A Type IV point corresponds to a chain of disks
\begin{equation}
\ol D(a_1,r_1) \supset \ol D(a_2,r_2) \supset \cdots
\end{equation}
with empty intersection.

One may define the open Berkovich disk $D^{\Berk}(a,r)$ simply by deleting the bars in the previous paragraph.  An open connected Berkovich affinoid in $\ol D^{\Berk}(a,r)$ is defined to be a finite intersection of open Berkovich disks and complements of closed Berkovich disks.  The set of open connected Berkovich affinoids forms a basis for a topology on the closed Berkovich disk known as the Gel'fand topology.

A natural question to ask is, when does a homeomorphism $h$ on $\ol D(a,r)$ extend to a homeomorphism on its Berkovich closure?  We settle the following (easy) case here, which is enough for our purposes.

\begin{prop} Suppose that $h: \ol D(a,r) \isom \ol D(a',r)$ is a bijective isometry.
Then $h$ extends uniquely to a homeomorphism of Berkovich disks
\begin{equation}
 h^{\Berk}: \ol D^{\Berk}(a,r) \isom \ol D^{\Berk}(a',r).
\end{equation}
\end{prop}
\begin{proof}
By hypothesis, $h$ is in particular a homeomorphism which takes open (resp. closed) disks to open (resp. closed) disks of the same radius.  The same is true for $h^{-1}$.

We define $h^\Berk$ on Type I points by $h^\Berk(a)=h(a)$.  On Type II and III points we define $h^\Berk$ by $h^\Berk(D(b,s))=h(D(b,s))$.  If $\ol D(a_1,r_1) \supset \ol D(a_2,r_2) \supset \cdots$ is a chain of disks with empty intersection, then the same is true for $h(\ol D(a_1,r_1)) \supset h(\ol D(a_2,r_2)) \supset \cdots$. So we define $h^\Berk$ by taking this chain of images.  Thus $h^\Berk$ is defined for all types of points in $D^{\Berk}(a,r)$.   Similarly we define
\begin{equation}
 (h^{-1})^{\Berk}: \ol D^{\Berk}(a',r) \isom \ol D^{\Berk}(a,r);
\end{equation}
clearly, $(h^{-1})^{\Berk}=(h^{\Berk})^{-1}$.  In particular, $h^{\Berk}$ is a bijection.
This operation is also compatible with restriction: if $D(b,s) \subset D(a,r)$, then it is easy to see that
\begin{equation}
h^{\Berk}| D(b,s)=(h|D(b,s))^{\Berk}.
\end{equation}
Thus the image of an open Berkovich disk under $h^{\Berk}$ is again an open Berkovich disk.  Similarly for closed Berkovich disks, and for open connected Berkovich affinoids.  It follows that $h^{\Berk}$ is an open map in the Gel'fand topology.  Similarly for $(h^{\Berk})^{-1}$ and therefore $h^{\Berk}$ is a homeomorphism.  The uniqueness follows from the density of Type I points in the Berkovich disk.
\end{proof}

\begin{cor} The conjugating map $h$ in Theorem \ref{lipeomorphic conjugacy} extends to the Berkovich disk.
\end{cor}

\begin{proof} The theorem implies in particular that $|\ol h(x_2)- \ol h(x_1)|< |x_2-x_1|$, thus $h$ is an isometry.
\end{proof}

\section{Orbit Counting and Pointwise Estimates of Conjugating Homeomorphisms}\label{section:s5}

In this section we make a close study of the $\Zp$-orbits of our functions, and how rapidly an intertwining map $h$ needs to shrink or expand them.  Our analysis will quickly recover the known fact that the functions $f_{m+1,a}$ are all homeomorphically equivalent.  A closer study will find precise H\"{o}lder exponents at $0$ for the intertwining maps $h$, as mentioned in the introduction.  Of course, if $h$ is a lipeomorphism then both exponents are (no less than) one.  For a cleaner theory, we consider the set of $\Zp$-orbits abstractly as a space which is discrete except for a clustering condition at $0$.

\subsection{Bullseye Spaces}

The term ``bullseye space'' is motivated by the orbit diagram in Figure 2, and the following idea.  To define a conjugating homeomorphism, one needs to match up $\Zp$-orbits for the two functions.  Since the $\Zp$-orbits are open, the only topological constraint on this is that, as the $\Zp$-orbits in one domain tend to the `bullseye' $0$, so should the matching $\Zp$-orbits in the other domain.

\begin{define} Let $(X_n)_{n \in \N}$ be a sequence of finite sets, so that $X_n$ is nonempty for $n$ sufficiently large.  Let $X$ be the disjoint union of the $X_n$, together with a distinguished element $\ast$.  Topologize $X$ as follows:  The open sets containing $\ast$ are those which contain $\bigcup_{n=N}^\infty X_n$ for some $N$.  Every set that does not contain $\ast$ is open.  Then $X$ is called a bullseye space, and the sets $X_n$ are called the rings of $X$. Let $X^\circ=X-\{\ast \}=\bigcup_{n=1}^\infty X_n$.  We define morphisms of bullseye spaces to be continuous maps sending $\ast$ to $\ast$.
\end{define}

In particular the subspace $X^\circ$ is discrete.  Suppose $Y$ is another bullseye space with rings $Y_n$.  Then a map $\beta: X^\circ \to Y^\circ$ may be extended to a morphism from $X$ to $Y$ if and only if:
\begin{equation} \label{continu}
 \text{ For every $M \in \N$ there is an $N \in \N$ so that } \beta \left( \bigcup_{n=N}^\infty X_n \right) \subseteq \bigcup_{n=M}^\infty Y_n.
\end{equation}

\begin{define} Let $X$ and $Y$ be infinite bullseye spaces with rings $X_n$ and $Y_n$. Put $x(n)=|X_n|$ and $y(n)=|Y_n|$.  We define a numerical function $\mu$ on $\N$ via
\begin{equation*}
\mu(N)=\min \left\{ M \in \N \mid \sum_{n=1}^{N-1} x(n) \leq \sum_{n=1}^{M} y(n) \right\}.
\end{equation*}
\end{define}
Since $Y$ is infinite, $\mu$ is defined.  Since $X$ is infinite, we have $\lim_{N \to \infty} \mu(N)= \infty$.

Note that for all $N$,
\begin{equation} \label{fill}
\sum_{n=1}^{\mu(N)-1} y(n) \leq \sum_{n=1}^{N-1} x(n) \leq \sum_{n=1}^{\mu(N)} y(n).
\end{equation}

\begin{prop} \label{reader} Any two infinite bullseye spaces are isomorphic. \end{prop}

\begin{proof}
 By (\ref{fill}) we may inductively construct injections
\begin{equation*}
\beta_N: \bigsqcup_{n=1}^{N-1} X_{n} \to \bigsqcup_{n=1}^{\mu(N)}Y_n
\end{equation*}
so that
\begin{enumerate}
\item
\begin{equation} \label{ichi}
\bigsqcup_{n=1}^{\mu(N)-1}Y_n \subseteq \beta_N \left(\bigsqcup_{n=1}^{N-1} X_{n} \right) \subseteq \bigsqcup_{n=1}^{\mu(N)}Y_n.
\end{equation}
\item $\beta_{N+1}$ restricted to $\bigsqcup_{n=1}^{N-1}X_{n}$ agrees with $\beta_N$.
\end{enumerate}
Write $\beta:X \to Y$ for the limit of these maps, and putting $\beta(\ast)=\ast$; it is an isomorphism.
\end{proof}

The previous proposition underlies the known fact that two functions which are tangent to the identity near a fixed point are homeomorphically equivalent.  To establish the finer results on H\"{o}lder exponents, we specialize the morphisms as follows.  (See Proposition \ref{betaprop}.)

\begin{define} Let $X$ and $Y$ be two bullseye spaces, and $\alpha >0$.  We say that a map $\beta:X \to Y$ with $\beta(\ast)=\ast$ is an $\alpha$-morphism if there exists $k \in \R$ so that if $x \in X_N$ and $\beta(x) \in Y_M$, then $M \geq \alpha N+k$.  Given $\alpha_1,\alpha_2 >0$, we say that $\beta$ is an $(\alpha_1,\alpha_2)$-isomorphism if $\beta$ is a bijection and an $\alpha_1$-morphism, with $\beta^{-1}$ an $\alpha_2$-morphism.   \end{define}

We record some consequences of the definitions.

\begin{lemma} \label{morphlemma}
\begin{enumerate}
\item Any $\alpha$-morphism is a morphism.
\item If $\alpha_1<\alpha_2$, then any $\alpha_2$-morphism is an $\alpha_1$-morphism.
\item If $\beta_1: X \to Y$ is an $\alpha_1$-morphism and $\beta_2: Y \to Z$ is an $\alpha_2$-morphism, then
$\beta_2 \circ \beta_1$ is an $\alpha_1 \alpha_2$-morphism.

\end{enumerate}
\end{lemma}

\begin{proof} For the first item, let $M \in \N$.  Pick $N \in \N$ so that $\alpha N+k \geq M$.  Then 
\begin{equation}
\beta \left( \bigcup_{n=N}^\infty X_n \right) \subseteq \bigcup_{n=M}^\infty Y_n.
\end{equation}
Thus $\beta$ is a morphism.  The other items are immediate.
\end{proof}

We are naturally concerned with the case of $0 < \alpha \leq 1$.  The following proposition gives a numerical criterion for when there exists an $(\alpha_1,\alpha_2)$-isomorphism:

\begin{prop} \label{positive} If there are $k_1,k_2 \in \R$ so that
 \begin{equation}
 \alpha_1 N+k_1 \leq \mu(N) \leq \alpha_2^{-1}N +k_2,
    \end{equation}
 for all $N$ sufficiently large, then there is an $(\alpha_1,\alpha_2)$-isomorphism from $X$ to $Y$.
\end{prop}

\begin{proof} Recall the homeomorphism $\beta$ from Proposition \ref{reader}.
Suppose that $N$ is sufficiently large.  Let $x \in X_N$ and $\beta(x) \in Y_M$.  By the construction of $\beta$, we have $\mu(N) \leq M \leq \mu(N+1)$. By hypothesis we have, then,
\begin{equation*}
\alpha_1N+k_1 \leq M \leq \alpha_2^{-1}N+k_2.
\end{equation*}
The result follows.
\end{proof}

\begin{lemma} \label{HHotel}
Let $X$ be a bullseye space and $F$ a finite subset.  Let $X_-=X-F$.  We view $X_-$ as a bullseye space with rings given by $(X_-)_i=X_i-F$.  Then there is a (1,1)-isomorphism $\xi: X \isom X_-$. \end{lemma}
\begin{proof}
It is easy to see that $N-1 \leq \mu(N) \leq N+k$ for some $k$.  The conclusion follows from Lemma \ref{morphlemma}.
\end{proof}

The next proposition gives a numerical criterion for when there does {\it not} exist a $(\alpha_1,\alpha_2)$-isomorphism:

\begin{prop} \label{fractions} Suppose that there are numbers $a,b$ so that for all $N$ sufficiently large,
\begin{equation} \label{Black Beauty}
 \sum_{n=1}^{bN} x(n) < \sum_{n=1}^{aN} y(n).
\end{equation}
Let $\alpha>\frac{a}{b}$. Then there does not exist a surjective $\alpha$-morphism $\beta: X \to Y$.
\end{prop}

\begin{proof}  Suppose $\beta: X \to Y$ is a surjective morphism.  Let $N$ be sufficiently large.  Then by the inequality (\ref{Black Beauty}), there exists $x_N \notin \cup_{i=1}^{bN}X_i$, but $\beta(x_N) \in \cup_{j=1}^{aN}Y_j$.
So $x_N \in X_i$ with $i>bN$, and $\beta(x_N) \in Y_j$ with $j \leq aN$.  If there were a constant $k$ so that $j \geq \alpha i-k$ for all $x_N$, then we would have
\begin{equation*}
aN \geq \alpha bN-k
\end{equation*}
for all $N$ sufficiently large.  Since $\alpha b >a$ this is impossible.
\end{proof}

\begin{prop} \label{constants} Let $n_0,x_0,y_0 \in \N$.  Suppose that for $n \geq n_0$, $x(n)=x_0$ and $y(n)=y_0$.  Then there exists an $(\frac{x_0}{y_0},\frac{y_0}{x_0})$-isomorphism $\beta:X \to Y$.  If $\alpha>\frac{x_0}{y_0}$, then there does not exist a surjective $\alpha$-morphism $\beta: X \to Y$.
\end{prop}

\begin{proof}
By Lemma \ref{HHotel} we may assume $n_0=1$.  Note that
\begin{equation*}
\mu(N)=\left \lceil \frac{(N-1)x_0}{y_0} \right \rceil,
\end{equation*}
where by $\lceil x \rceil$ we denote that smallest integer greater or equal to $x$.  Then Proposition \ref{positive} gives the first statement.
For the second statement, pick a rational number $\frac{a}{b}$ with $\alpha > \frac{a}{b} >\frac{x_0}{y_0}$.  Thus $ay_0>bx_0$, and now Proposition \ref{fractions} gives the result.
\end{proof}

\begin{prop} \label{mixed} Let $n_0,x_0,y_0,r \in \N$, with $r \geq 2$.  Suppose that for $n \geq n_0$,  $x(n)=x_0$ and $y(n)=y_0 r^n$.  Then there does not exist a surjective $\alpha$-morphism $\beta: X \to Y$ for any $\alpha>0$. \end{prop}

\begin{proof}
Again we may assume $n_0=1$.  Say $\alpha> \frac{1}{b}$ for some $b \in \N$.  Set $a=1$.  Considering $N$ large enough so that
\begin{equation*}
x_0bN<y_0r^N < y_0 \sum_{n=1}^{N}r^n,
\end{equation*}
we apply Proposition \ref{fractions}.
\end{proof}

\begin{prop} \label{harder} Let $n_0,\lambda,\nu,x_0,y_0,r$ be natural numbers with $1 < \lambda \leq \nu $ and $r \geq 2$.  Suppose that for $n \geq n_0$, $x(n)=x_0r^{\nu n}$ and $y(n)=y_0 r^{\lambda n}$.  Then then there exists a $(1,\frac{\lambda}{\nu})$-isomorphism $\beta:X \isom Y$.  If $\alpha>\frac{\lambda}{\nu}$, then there does not exist a surjective $\alpha$-morphism $\beta: Y \to X$. \end{prop}

\begin{proof}
As before we may assume $n_0=1$.

  For the first statement, we need estimates on $\mu(N)$. For the upper bound we argue that there is a $k_2 \in \R$ so that $\mu(N) \leq \frac{\nu}{\lambda}+k_2$ for $N$ large. By summing the geometric series, we have
  \begin{equation*}
  \mu(N)= \min \left(M \mid x_0 \frac{r^{N\nu}-r^{\nu}}{r^\nu-1} \leq y_0 \frac{r^{M \lambda}-r^\lambda}{r^\nu-1} \right).
  \end{equation*}
  The inequality may be written as
  \begin{equation}
  C_1 r^{N \nu}+D_1 \leq C_2 r^{M \lambda}+ D_2
  \end{equation}
  for constants $C_1,C_2,D_1,D_2$ with $C_1$ and $C_2$ positive. If we consider integers $M$ of the form $M= \frac{\nu}{\lambda}N+k_2$, and put $R=r^\nu$ and $k=k_2 \lambda$ this is simply
  \begin{equation}
  C_1 R^N+D_1 \leq C_2 R^{N+k}+D_2.
  \end{equation}
  Clearly we can find $k$ large enough so that this is satisfied for large $N$.

  For the lower bound, note that for $N$ large we have
  \begin{equation*}
  C_2 r^{(N-1)\lambda}+D_2 < C_1 r^{N \nu}+D_1.
\end{equation*}
Thus we have $N \leq \mu(N) \leq \frac{\nu}{\lambda}N+k_2$ for $N$ sufficiently large, and the conclusion follows from Proposition \ref{positive}.

For the second statement, pick a rational number $\frac{a}{b}$ with $\alpha> \frac{a}{b} > \frac{\lambda}{\nu}$.  Then $a \nu > \lambda b$.  We claim that for $N >>0$,
\begin{equation*}
\sum_{n=1}^{bN} y(n) < \sum_{n=1}^{aN} x(n),
\end{equation*}
whence the second statement.  Summing the geometric series we must show
\begin{equation*}
y_0 \cdot \frac{r^{\lambda(bN+1)}-1}{r^\lambda-1} < x_0 \cdot\frac{r^{\nu(aN+1)}-1}{r^\nu-1}.
\end{equation*}
To clarify the situation, let $R_1=r^{\lambda b}$ and $R_2=r^{\nu a}$.  Note that $1<R_1<R_2$.  Then the inequality follows from the general fact that given constants $C_1,C_2,D_1,D_2$, with $C_1,C_2$ positive, we have
\begin{equation*}
C_1R_1^N+D_1 < C_2 R_2^N+D_2
\end{equation*}
for $N >>0$.
\end{proof}

\subsection{Application}

Suppose in this section that for $i=1,2$, we have neighborhoods $U_i$ of $0 \in \Qp$ and homeomorphisms $f_i: U_i \isom U_i$ satisfying the following conditions:
\begin{enumerate}
\item $f_i(0)=0$.
\item Each $U_i$ is $f_i$-invariant.
\item The functions $f_i$ are norm-preserving on the $U_i$.
\item For $x \in U_i$, the function $\varphi: \Z \to U_i$ given by $\varphi(z)=f^{\circ z}(x)$ is Lipschitz (giving $\Z$ the $p$-adic metric).  We therefore have $\Zp$-orbits of $f$ in $U_i$.
\item \label{open} Each $\Zp$-orbit of $f$ in $U_i-\{0 \}$ is open.
\item The $U_i$ are bounded.
\end{enumerate}

Condition (iv) is equivalent to the condition that there are only finitely many $\Zp$-orbits of $f$ in each circle of constant radius.  This is because each circle is compact and $\Zp$-orbits are necessarily compact, thus closed.

 For $U \subseteq \Qp$, let $[U]$ denotes the set of $\Zp$-orbits $\orb$ in $X$.  Let $n_i= \min \{n|\pp^n \cap U_i \neq \phi \}-1$.  Set $X=[U_1]$, $Y=[U_2]$; then $X$ and $Y$ are bullseye spaces with rings $X_i=[U_1 \cap p^{n_1+i}\Zp^\times]$, and $Y_j=[U_2 \cap p^{n_2+j} \Zp^\times]$.

\begin{define} Given a homeomorphism $h: U_1 \isom U_2$ which intertwines $f_1$ and $f_2$, define the quotient map $\beta=\beta(h): X \isom Y$ via $\beta(\orb)=h(\orb)$. \end{define}

Since $h(0)=0$ we have $\beta(\ast)=\ast$.

\begin{prop} \label{betaprop} This defines an isomorphism $\beta(h): X \isom Y$.  Conversely, any isomorphism $\beta: X \isom Y$ is the quotient $\beta=\beta(h)$ of a homeomorphism $h: U_1 \isom U_2$ intertwining $f_1$ and $f_2$.  Moreover, $\beta(h)$ is an $\alpha$-morphism if and only if there is a constant $C$ so that
\begin{equation} \label{Holder}
|h(x)| \leq C |x|^\alpha
\end{equation}
for $x \in U_1$.
\end{prop}

\begin{proof} The continuity of $h$ at $0$ ensures that $\beta(h)$ is continuous, and similarly for $\beta(h)^{-1}$.
In the second statement, we are given an isomorphism $\beta$.  We define $h$ by putting $h(0)=0$ and defining it continuously on each of the disjoint open $\Zp$-orbits in $U_1- \{0 \}$.  Given such an orbit $\orb$, pick any $x \in \orb$ and $y \in h(\orb)$ and then define $h$ on $\orb$ via $h( f_1^{\circ z}(x))=f_2^{\circ z}(y)$, for $z \in \Zp$.
Then condition (\ref{continu}) implies that $h$ is continuous at $0$, and it is clear that this construction intertwines the $f_i$.
Finally, suppose that (\ref{Holder}) holds.  We show that $\beta(h)$ is an $\alpha$-morphism.  Let $\orb \in X_i$, and pick $a \in \orb$.  Then $\beta(\orb) \ni h(a) \in Y_j$.  Equation (\ref{Holder}) translates to
 \begin{equation*}
 p^{-(n_2+j)} \leq p^{-(n_1+i)+\kappa},
 \end{equation*}
 where $\kappa=\log_pC$.  Thus $n_2+j \geq (n_1+i)\alpha-\kappa$, which shows that $\beta$ is an $\alpha$-morphism.  The reverse direction is similar.
 \end{proof}

\begin{prop} \label{topconj}
Suppose that $f_i$ and $U_i$ satisfy the conditions introduced at the beginning of this section.
Then there is a homeomorphism $h: U_1 \isom U_2$ so that $f_2 \circ h=h \circ f_1$ on $U_1$.
\end{prop}

\begin{proof} This follows from Propositions \ref{reader} and \ref{betaprop}. \end{proof}

We now have all of the tools necessary to prove Theorem \ref{local Holder continuity}.

\begin{proof} By Theorem \ref{lipeomorphic conjugacy} we may assume $f=f_{m+1,a}$ and $g=f _{m'+1,a'}$.  We may apply Proposition \ref{topconj}; the functions are norm-preserving close to $0$ by the remark after Equation (\ref{flowing}).
Taking $\Zp$-orbits then puts us in the category of bullseye spaces.  By Proposition \ref{counting orbits} there are $x_0 p^{im}$, resp. $y_0 p^{im'}$ orbits in each circle $C(0,p^{-i})$ for positive constants $x_0,y_0$.  Proposition \ref{harder} tells us that our best $\alpha$-morphism is for $\alpha=\frac{m}{m'}$.  Combined with Proposition \ref{betaprop}, this gives the growth condition (\ref{parade}), and the subsequent statement.

Next, suppose that $y=f^{\circ z}(x)$ for $z \in \Zp$.  Then
\begin{equation*}
|y-x|=|azx^{m+1}|
\end{equation*}
by Proposition \ref{Qp-orbits}.
Similarly,
\begin{equation*}
|h(y)-h(x)|=|g^{\circ z}(h(x))-h(x)|=|a'zh(x)^{m'+1}|.
\end{equation*}
Together with the inequality (\ref{parade}) this gives
\begin{equation*}
|h(y)-h(x)| \leq C' |y-x|^{\frac{m}{m+1} \cdot \frac{m'+1}{m'}},
\end{equation*}
for some constant $C'$ which is independent of $x$ and $y$, and $z$.  (Recall $|z| \leq 1$.)
Since
\begin{equation*}
\frac{m}{m+1} \cdot \frac{m'+1}{m'} \geq \frac{m}{m'},
\end{equation*}
this gives (\ref{furrst}).

The other direction is similar.  Now go back and let $C$ be the minimum of all these constants.
\end{proof}

\section{Indifferent Multiplier Maps in $\Qp$} \label{Multipliers}

We now turn to multiplier maps $L_a(x)=ax$ for $a \in \Qp^\times$.
Of course, if $a \neq a'$ then $L_a$ is not analytically equivalent to $L_{a'}$.  On the other hand, any two contracting (resp. expanding, resp. irrational indifferent) multiplier maps are homeomorphically equivalent.  (If $a$ is a primitive $k$th root of unity, then $L_a$ is homeomorphically equivalent to $L_{a'}$ if and only if $a'$ is also a primitive $k$th root of unity.)

In this section we will determine lipeomorphic equivalence of these maps; in particular, we find the best H\"{o}lder estimate at $0$ for a conjugating homeomorphism $h$.  We first treat the simple case of contracting multiplier maps.

\begin{prop} \label{utility}
For $i=1,2$, let $p_i$ be primes, $\pi_i \in p_i \Z_{p_i}- \{ 0 \}$ and write $L_{\pi_i}: \Q_{p_i} \to \Q_{p_i}$ for dilation by $\pi_i$. Then there is a homeomorphism $h: \Q_{p_1} \isom \Q_{p_2}$ so that $h \circ L_{\pi_1} \circ h^{-1}=L_{\pi_2}$.   In fact, we may choose $h$ and a constant $C>0$ with the property that
\begin{equation*}
|h(x)| \leq C|x|^{\frac{\ord(\pi_1)}{\ord(\pi_2)}}.
\end{equation*}
This exponent cannot be improved.
\end{prop}

Of course the case of expanding multiplier maps can be deduced from this.

Recall the definition of $\Delta_\pi$ from Section \ref{decomposition}.

\begin{proof}
$\Delta_{\pi_i}$ is homeomorphic to Cantor's triadic set, being compact, totally disconnected, and a perfect metric space.  Let $h_0: \Delta_{\pi_1} \isom \Delta_{\pi_2}$ be a homeomorphism.  We may define $h$ by $h(0)=0$, and $h(\pi_1^n u)=\pi_2^n h_0(u)$ for $u \in \Delta_{\pi_1}$ and $n \in \Z$.  The statement about exponents is straightforward.
\end{proof}

Note that homeomorphisms exist between contracting maps defined on different $p$-adic fields.  In the next section we will see that this phenomenon is special to this case.

Let us turn now to the more interesting case of indifferent multiplier maps.  The main point about indifferent multiplier maps is that they have but finitely many $\Zp$-orbits at every radius.  It takes only a little more work to determine the precise number of these orbits, so we include this calculation as well.

\begin{prop} Let $a \in \Zp^{\times}$, not a root of unity.
\begin{enumerate}
\item If $p \neq 2$, then the compact subgroups of $1+ \pp$ are the groups $1+\pp^n$, $n \geq 1$.  For all $p$, the compact subgroups of $1+ \pp^2$ are the groups $1+\pp^n$, $n \geq 2$.

\item If $a \in 1+ \pp$, then the map $n \mapsto a^n$ for $n \in \N$ is uniformly continuous (for the $p$-adic topology on $\N$) and thus extends to a continuous homomorphism $\varphi(z)= a^z$ for $z \in \Zp$.

\item Let $a \in 1+\pp$.  Suppose that either $p \neq 2$ or $a \in 1+\pp^2$.  Then $\{a^z \mid z \in \Zp \}=1+\pp^n$, where $n= \ord(a-1)$.
\item Let $a \in \Zp^\times$ with $p \neq 2$.  Write $a=\zeta a_1$, where $a_1 \in 1+\pp$ and $\zeta$ is a $k$th root of unity, with $p \nmid k$.  Then $\{a^z \mid z \in \Zp \}=(1+\pp^n) \times <\zeta>$, where $n= \ord(a_1-1)$.

\item Suppose that $p=2$ and $a \in (1+\pp)-(1+\pp^2)$.  Then $\{ a^z \mid z \in \Zp \}=(1+\pp^{n+1}) \cup (a+\pp^{n+1})$, where $n=\ord(a+1)$.
\end{enumerate}
\end{prop}

\begin{proof}
We omit the proof of (i). Item (ii) may be found in \cite{Koblitz}.  For (iii), note that the image of $\varphi$ is a compact subgroup of $1+\pp$, and use the first part.  For (iv), note that
\begin{equation*}
\{a^z \mid z \in \Zp \} \supseteq \{(a_1^k)^z \mid z \in \Zp \}=1+\pp^n \ni a_1.
\end{equation*}
This implies that $\zeta \in  \{a^z \mid z \in \Zp \}$, giving the result.
For (v), write $a=-1+2^nu$ where $u$ is a unit in $\Z_2$.  Note that $\ord(a^2-1)=n+1$, and so by part (iii), $\{ a^{2z} \mid z \in \Zp \}=1+ \pp^{n+1}$.  Multiplying this subgroup by $a$ gives the ``odd'' part.
\end{proof}

\begin{define} Let $a \in  \Zp^\times$, not a root of unity, and write $L_a: \Qp \to \Qp$ for the map $x \mapsto ax$.  Write $N(a)$ for the number of $\Zp$-orbits for $L_a$ in $\Zp^\times$.
\end{define}

\begin{cor}  \label{multorbs}

\begin{enumerate}

\item  Suppose that $p \neq 2$, and let $a=\zeta a_1$, where $a_1 \in 1+ \pp$ and $\zeta$ is a primitive $k$th root of unity, with $p \nmid k$. Then $N(a)=\frac{1}{k}(p-1) p^{n-1}$, where $n=\ord(a_1-1)$.
\item Suppose that $p=2$ and $a \in 1+ \pp^2$.  Then $N(a)=2^{n-1}$, where $n=\ord(a-1)$.
\item Suppose that $p=2$ and $a \in (1+\pp)-(1+\pp^2)$. Then $N(a)=2^{n-1}$, where $n=\ord(a+1)$.
\item \label{giga} In all cases, the number of $\Zp$-orbits for $L_a$ in each circle of constant norm               is exactly $N(a)$.
\end{enumerate}
\end{cor}

For example, let $p=5$ and $a=4$.  Then $4=(-1)(-4)$, so here $k=2$ and $n=1$.  It follows that $N(4)=2$.  On the other hand, $N(2)=1$ and indeed $\orb(L_2)=\Z_5^\times$.

\begin{prop}  If $a,a' \in \Zp^{\times}$ are not roots of unity, then $L_a$ and $L_{a'}$ are homeomorphically equivalent.  They are lipeomorphically equivalent if and only if $N(a)=N(a')$.  More specifically, suppose that $N(a) \leq N(a')$.  Then we may choose $h$ and a constant $C>0$ with the property that
\begin{equation*}
|h(x)| \leq C|x|^{\frac{N(a)}{N(a')}}.
\end{equation*}
This exponent cannot be improved.
\end{prop}

\begin{proof}
This follows from Proposition \ref{betaprop}, with the best $\alpha$ computed using Proposition \ref{constants}.  The number of orbits in each circle is constant by Corollary \ref{multorbs} (iv) above.
\end{proof}

Continuing with the example in $\Q_5$, the best possible H\"{o}lder constant for a homeomorphism between $L_2$ and $L_4$ is $\alpha=\half$.  In particular $L_2$ and $L_4$ are not lipeomorphically equivalent as functions on $\Q_5$.

The following proposition regards the linearization of mappings $f$ which are tangent to the identity.  In fact, no local conjugating map can satisfy the H\"{o}lder estimate at $0$, and in particular $f$ is not lipeomorphically linearizable.

\begin{prop}\label{linearization} Let $f(x)=x+ax^{m+1}+ \cdots$ be a $\Qp$-locally-analytic function with $a \neq 0$, defined on a neighborhood $U$ of $0$. Let $b \in \Zp^\times$, not a root of unity.  Then, there is a neighborhood $U'\subseteq U$ of $0$ and a homeomorphism $h$ conjugating $f(x)$ to $L_b(x)=bx$. However, it is not possible to choose constants $C>0$ and $\alpha >0$ so that, for all $x\in U'$, $|h(x)|\leq C|x|^{\alpha }$. In other words, the homeomorphism $h$ cannot be chosen to be H\"{o}lder continuous in any neighborhood of $0$. \end{prop}

\begin{proof}
By Theorem \ref{lipeomorphic conjugacy}, we may assume that $f=f_{m+1,a}$ for some $m \geq 1$ and $a \in \Zp - \{0 \}$.
By Proposition \ref{counting orbits}, the number of $\Zp$-orbits for $f$ in $C(0,p^i)$ is of the form $y_0p^{im}$ for some constant $y_0$.  As is seen in the discussion above, the number of $\Zp$-orbits for $L_b$ in $C(0,p^i)$ is a finite number which does not depend on $i$.  By Proposition \ref{topconj}, there is a homeomorphism conjugating $f$ to $L_b$, but by Propositions \ref{mixed} and \ref{betaprop}, it cannot satisfy any H\"{o}lder estimate at $0$.
\end{proof}

\section{Homeomorphic Equivalence Amongst Different Fields}\label{section:s6}

In this section, we settle the following question.  We have seen in Proposition \ref{utility} that it is possible for contracting multiplier maps on two different $p$-adic fields to be homeomorphically equivalent.  Can this happen for analytic maps tangent to the identity on two different $p$-adic fields?  At first glance it may seem plausible, since $\Qp$ and $\Q_q$ are homeomorphic for $p \neq q$.  However we will see this possibility is precluded by the different manners in which their iterates accumulate.

\begin{lemma} \label{Euclid} Let $p \neq q$ be prime numbers, and $(X,d)$ a metric space.  Suppose that for a bijection $f: X \to X$,
\begin{equation*}
\lim_{n \to \infty} f^{\circ kp^n}(x)=x \text{ and } \lim_{n \to \infty} f^{\circ kq^n}(x)=x
\end{equation*}
uniformly for all $k \in \Z$ and $x \in X$.  Then $f(x)=x$.
\end{lemma}

\begin{proof} Let $\eps >0$.  Pick $n$ large enough so that
\begin{equation*}
d(f^{\circ kp^n}(x),x) < \eps \text{ and } d(f^{\circ kq^n}(x),x) < \eps.
\end{equation*}
Since $p^n$ and $q^n$ are relatively prime, we may choose $a,b \in \Z$ so that $ap^n+bq^n=1$.  Then for any $x \in X$ we have
\begin{equation*}
\begin{split}
d(f(x),x) &=d(f^{\circ (ap^n+bq^n)}(x),x) \\
         &=d(f^{\circ ap^n}(f^{\circ bq^n}(x)),x) \\
         &\leq d(f^{\circ ap^n}(f^{\circ bq^n}(x)),f^{\circ bq^n}(x))+d(f^{\circ bq^n}(x),x) \\
         &<2 \eps. \\
\end{split}
\end{equation*}
This being true for all $\eps>0$, we conclude that $f(x)=x$.
\end{proof}

\begin{cor} \label{redact} Let $p_1 \neq p_2$ be prime numbers.  For $i=1,2$, let $K_i$ be $p_i$-adic fields, $m_i \geq 2$, and $a_i \in K_i$.  Then the flows $f_i=f_{m_i,a_i}: K_i \to K_i$ are not locally homeomorphic.  That is, there is no local homeomorphism $h: U_1 \to U_2$ so that $h \circ f_1=f_2 \circ h$, where $U_i$ are $f_i$-invariant neighborhoods of $0$ in $K_i$.
\end{cor}

\begin{proof}  Note that
\begin{equation*}
\lim_{n \to \infty} f_i^{\circ kp_i^n}(x)=x
\end{equation*}
uniformly for all $k$ and $x$ by Proposition \ref{phis}.
Suppose there were such a homeomorphism.  Let $k \in \Z$.  For all $x \in U_2$ we have
\begin{equation*}
\begin{split}
\lim_{n \to \infty} f_2^{\circ kp_1^n}(x) &= \lim_{n \to \infty} h \circ f_1^{\circ kp_1^n} \circ h^{-1}(x) \\
                                          &= x. \\
\end{split}
\end{equation*}
By Lemma \ref{Euclid}, it follows that $f_2(x)=x$, a contradiction.
\end{proof}

\noindent {\bf Remark:}  It similarly follows from Lemma \ref{Euclid} that indifferent (and nontrivial) multiplier maps $L_{a_i}$ on different $p_i$-adic fields are not locally homeomorphic.

\begin{prop}
 Let $p_1 \neq p_2$ be prime numbers.  For $i=1,2$, let $K_i$ be $p_i$-adic fields, and $f_i(x)=x+O(x^2)$, with $f_i(x) \neq x$.   Then the $f_i$ are not locally homeomorphic.  That is, there is no local homeomorphism $h: U_1 \to U_2$ so that $h \circ f_1=f_2 \circ h$, where $U_i$ are $f_i$-invariant neighborhoods of $0$ in $K_i$.
\end{prop}

\begin{proof}
This follows from Corollary \ref{redact}.
\end{proof}

\section{Appendix: Lipschitz Maps under change of variable} \label{appendix}
We show here that a Lipschitz condition in a neighborhood of $\infty $ may be transported to a neighborhood of $0$, under suitable hypotheses. This result is needed for the proof of Theorem \ref{curbed}. In the case that the field $K={\mathbf C}$, this fact is cited in the paper of Shcherbakov \cite{Shch}. We have provided this proof solely for the reader's convenience, claiming no originality. Although the statement given in Section \ref{section:s4} is for non-archimedean fields, the method of the proof below also adapts to the complex case.

\begin{prop} \label{Americanproof}\label{Lipschitz at 0}
Let $K$ be a $p$-adic field, and $\pi \in \pp$ a nonzero element.  Let $m,N \geq 1$.

Let $Y=\A_{m+1}(\pi^N \Delta')$ and $\tilde h: Y \to Y$ a homeomorphism.  Let $\hat{h}(\eta)=\tilde h(\eta)-\eta$.  Suppose that for any $\eps'>0$, $\delta'>0$, there is an $m$-neighborhood $Y' \subset Y$ so that
\begin{enumerate}
\item $\tilde h$ restricts to a homeomorphism on $Y'$.
\item For all $\eta \in Y'$, we have $|\hat{h}(\eta)| < \eps'$.
\item For all $\eta_1,\eta_2 \in Y'$, we have $|\hat{h}(\eta_2)-\hat{h}(\eta_1)| \leq \delta'|\eta_2-\eta_1|$.
\end{enumerate}

Let $\zeta \in \mu_m(K)$, and $X_\zeta=\{ \pi^{N}, \pi^{N+1},\ldots \} \cdot U_\zeta$.  Let
\begin{equation*}
h_\zeta= (\A_{m+1}^{-1})_\zeta \circ \tilde h \circ \A_{m+1} : X_\zeta \to X_\zeta,
\end{equation*}
and $\ol{h}_\zeta(x)=h_\zeta(x)-x$.

Then for all $\epsilon >0 $ and $\delta >0$, there is an $ N' \geq N$ so that
\begin{enumerate}
\item $h_\zeta$ restricts to a homeomorphism on $X_\zeta^\prime= \{ \pi^{N'}, \pi^{N'+1},\ldots\} \cdot U_\zeta$.
\item For all $x \in X_\zeta^\prime$, we have $|\ol{h}_\zeta(x)|\leq |x|^{m+1-\epsilon }$.
\item For all $x,y \in X_\zeta^\prime$, we have $|\ol{h}_\zeta(x)-\ol h_\zeta(y)|\leq \delta|x-y|$.
\end{enumerate}
\end{prop}

\begin{proof}
Note that $h_\zeta(x)=\zeta \cdot h_1(x)$, so we may as well assume that $\zeta=1$.  Let $h=h_1: X_1 \to X_1$ and $\ol h=\ol h_1$.

The map $h$ can be written as
\begin{equation*}
h(x)=\frac{x}{\root{m}\of{1-mx^{m}\hat{h}\left( -\frac{1}{mx^{m}}\right)}},
\end{equation*}
where the $m$th root in the denominator is as in Proposition \ref{Hensel}.  Note that for $x \in X_1$, the right hand side is in $X_1 \cdot (1+ m\pp)$.  By the construction of $U_1$, we have $(1+m \pp)U_1=U_1$ and therefore the right hand side is again in $X_1$.  Thus we have chosen the correct $m$th root.

We write the function $(1-u)^{-\frac{1}{m}}=1+F(u)$, where $F$ is an analytic function defined on the ball $B(0,R)$ for sufficiently small $R$. Moreover, $F$ satisfies the following estimates:
\begin{enumerate}
\item
$|F(u)|=O(u)$, and in particular, $|F(u)|\leq C_{R}$, where $C_{R}\rightarrow 0$ as $R\rightarrow 0$;
\item
$|F'(u)|<2$ if $R$ is sufficiently small.
\end{enumerate}
In particular, $F$ can be made as small as desired if $R$ is shrunk, and is a Lipschitz function with constant no greater than $2$.

The function $h$ can now be written
\begin{equation*}
h(x)=x+xF\left( mx^{m}\hat{h}\left( -\frac{1}{mx^{m}}\right) \right) ,
\end{equation*}
and so
\begin{equation*}
\ol h(x)= xF\left( mx^{m}\hat{h}\left( -\frac{1}{mx^{m}}\right) \right) .
\end{equation*}

 In view of this equation, given $\eps >0$, we may shrink $R$ so that $|\ol h(x)|\leq |x|^{m+1-\epsilon }$. We show now that $\ol h$ satisfies a Lipschitz condition.
\begin{equation*}\begin{split}
|\ol h(x)-\ol h(y)| &=\left| xF\left( mx^{m}\hat{h}\left( -\frac{1}{mx^{m}}\right) \right) -yF\left( my^{m}\hat{h}\left( -\frac{1}{my^{m}}\right) \right) \right| \\
&\leq \left| F\left( mx^{m}\hat{h}\left( -\frac{1}{mx^{m}}\right) \right) \right| |x-y|\\
&+|y|\left| F\left( mx^{m}\hat{h}\left( -\frac{1}{mx^{m}}\right) \right) -F\left( my^{m}\hat{h}\left( -\frac{1}{my^{m}}\right) \right) \right|\\
& \leq C_{R}|x-y|+2|y|\left| mx^{m}\hat{h}\left( -\frac{1}{mx^{m}}\right) -my^{m}\hat{h}\left( -\frac{1}{my^{m}}\right) \right|.
\end{split}
\end{equation*}

At this point, we will assume that $m=2$. The method for the general case should be clear. We will focus on the term
\begin{equation*}\begin{split}
&2|y|\left| 2x^{2}\hat{h}\left( -\frac{1}{2x^{2}}\right) -2y^{2}\hat{h}\left( -\frac{1}{2y^{2}}\right) \right| \\
&= 2|2y|\left|x^{2}\hat{h}\left( -\frac{1}{2x^{2}}\right) -x^{2}\hat{h}\left( -\frac{1}{2xy}\right)+x^{2}\hat{h}\left( -\frac{1}{2xy}\right)-y^{2}\hat{h}\left( -\frac{1}{2y^{2}}\right) \right| \\
&\leq 2|2x^{2}y|\left| \hat{h}\left( -\frac{1}{2x^{2}}\right)-\hat{h} \left( -\frac{1}{2xy}\right) \right| +2|2y|\left| x^{2}\hat{h}\left( -\frac{1}{2xy}\right) -y^{2}\hat{h}\left( -\frac{1}{2y^{2}}\right)\right| \\
&\leq 2\delta '|x-y|+2|2y|\left| x^{2}\hat{h}\left( -\frac{1}{2xy}\right) -y^{2}\hat{h}\left( -\frac{1}{2y^{2}}\right)\right|
\end{split}
\end{equation*}

This reduces us to estimating the term
\begin{equation*}
\begin{split}
&2|2y|\left| x^{2}\hat{h}\left( -\frac{1}{2xy}\right) -y^{2}\hat{h}\left( -\frac{1}{2y^{2}}\right)\right| \\
&= 2|2y|\left| x^{2}\hat{h}\left( -\frac{1}{2xy}\right) -xy\hat{h}\left( -\frac{1}{2xy}\right) +xy\hat{h}\left( -\frac{1}{2xy}\right) -y^{2}\hat{h}\left( -\frac{1}{2y^{2}}\right)\right| \\
&\leq 2|2xy|\left| \hat{h}\left( -\frac{1}{2xy}\right) \right||x-y| +2|2y^{2}|\left|x \hat{h}\left( -\frac{1}{2xy}\right) -y\hat{h}\left( -\frac{1}{2y^{2}}\right) \right| \\
&\leq D^{1}_{R}|x-y|+2|2y^{2}|\left| x\hat{h}\left( -\frac{1}{2xy}\right) -y\hat{h}\left( -\frac{1}{2y^{2}}\right) \right|.\\
\end{split}
\end{equation*}
In the last inequality, $D^{1}_{R}$ is a constant satisfying $D^{1}_{R}\rightarrow 0$ as $R\rightarrow 0$. Here, we have used the fact that $|\hat{h}|\leq \epsilon '$.

Repeating the technique above, we obtain
\begin{equation*}
\begin{split}
&2|2y^{2}|\left| x\hat{h}\left( -\frac{1}{2xy}\right) -y\hat{h}\left( -\frac{1}{y^{2}}\right) \right| \\
&\leq 2|2xy^{2}|\left| \hat{h}\left( -\frac{1}{2xy}\right) - \hat{h}\left( -\frac{1}{2y^{2}}\right) \right| +2|2y^{2}|\left| x\hat{h}\left( -\frac{1}{2y^{2}} \right) -y\hat{h}\left( -\frac{1}{2y^{2}}\right) \right| \\
&\leq 2\delta '|x-y|+ D^{2}_{R}|x-y|,
\end{split}
\end{equation*}
where again, $D^{2}_{R}\rightarrow 0$ as $R\rightarrow 0$. Thus, we have shown that $\hat h$ has a Lipschitz constant no greater than $C_{R}+4\delta '+D^{1}_{R}+D^{2}_{R}$ (and for general $m$, it is now apparent that the Lipschitz constant will be no greater than $C_{R}+2m\delta '+\sum_{i=1}^{m}D_{R}^{i}$). Now, given $0< \delta < 1$, we choose an $m$-neighborhood $Y$ of $\infty$ so that this sum is less than $\delta $.

\end{proof}

\bibliographystyle{plain}

\end{document}